\documentclass[12pt,a4paper]{amsart}
 \usepackage[foot]{amsaddr}
 \usepackage{appendix}
\makeatletter
\usepackage[english]{babel}
\usepackage[utf8]{inputenc}
\usepackage[square,sort,comma,numbers]{natbib}
\usepackage{graphicx}
\usepackage{amsfonts}
\usepackage{amsthm}
\usepackage{amsmath,amssymb}
\usepackage{xparse}
\usepackage{tikz-cd}
\usepackage{verbatim}
\usepackage{mathtools}
\usepackage[new]{old-arrows}
\setcitestyle{square}

\title{Extension DGAs}
\author{Haldun Özgür Bayındır}
\address{Department of Mathematics,
Institut Galilée,
University Paris 13,
99 av. JB Clément,
FR-93430 Villetaneuse, France}
\email{ozgurbayindir@gmail.com}

\NewDocumentCommand{\tens}{t_}
 {%
  \IfBooleanTF{#1}
   {\tensop}
   {\otimes}%
 }
\NewDocumentCommand{\tensop}{m}
 {%
  \mathbin{\mathop{\otimes}\displaylimits_{#1}}%
 }

\newcommand*{\rom}[1]{\expandafter\@slowromancap\romannumeral #1@}

\newcommand{\z}{\mathbb{Z}}
\newcommand{\Z}{\mathbb{Z}}

\newcommand{\F}{\mathbb{F}}

\newcommand{\DGA}{\mathcal DGA}
\newcommand{\I}{\mathbb{I}}
\newcommand{\ic}{\mathbb{I}_{\mathcal{C}}}
\newcommand{\id}{\mathbb{I}_{\mathcal{D}}}
\newcommand{\dsa}{\mathcal{A}_*}
\newcommand{\lv}{\lvert}
\newcommand{\wdg}{\wedge}
\newcommand{\rv}{\rvert}
\newcommand{\fp}{\mathbb{F}_p}
\newcommand{\sph}{\mathbb{S}}
\newcommand{\cc}{\mathcal{C}}
\newcommand{\dd}{\mathcal{D}}
\newcommand{\hfp}{H\mathbb{F}_p}
\newcommand{\hz}{H\mathbb{Z}}
\newcommand{\hfps}{H{\mathbb{F}_p}_*}
\DeclareMathOperator{\thh}{\ensuremath{\textup{THH}}}
\newcommand{\pis}{\pi_*}

\newcommand{\cof}{\rightarrowtail}
\newcommand{\lcof}{\leftarrowtail}

\newcommand{\we}{\smash{\rlap{\kern 6pt\raise 4pt\hbox{\footnotesize $\sim$}}}\rightarrow}
\newcommand{\bwe}{\smash{\rlap{\kern 8.5pt\raise 4pt\hbox{\footnotesize $\sim$}}}\longleftarrow}

\def\longfib{\DOTSB\relbar\joinrel\twoheadrightarrow}

\newcommand{\trfib}{\smash{\rlap{\kern 7pt\raise 4pt\hbox{\footnotesize $\sim$}}}\longfib}

\newcommand{\trcof}{\smash{\rlap{\kern 5.5pt\raise 4pt\hbox{\footnotesize $\sim$}}}\cof}

\newcommand{\alg}{\mbox{-} \mathcal Alg}
\newcommand{\calg}{ \mbox{-} c\mathcal Alg}
\newcommand{\Sp}{\mathbb{S}}
\mathchardef\mhyphen="2D
\newtheorem{theorem}{Theorem}[section]
\newtheorem{lemma}[theorem]{Lemma}

\newtheorem{corollary}[theorem]{Corollary}
\newtheorem{proposition}[theorem]{Proposition}

\theoremstyle{definition}

\newtheorem{definition}[theorem]{Definition}
\newtheorem{example}[theorem]{Example}
\newtheorem{remark}[theorem]{Remark}
\newtheorem{construction}[theorem]{Construction}

\def\co{\colon\thinspace}

\date{}

\usepackage{natbib}
\usepackage{graphicx}

\begin{document}
\title{Extension DGAs and topological Hochschild homology}
\maketitle
\begin{abstract}
     In this work, we study those differential graded algebras (DGAs) that arise from ring spectra through the extension of scalars functor. Namely, we study  DGAs whose corresponding Eilenberg-Mac Lane ring spectrum is equivalent to $H\mathbb{Z} \wedge E$ for some ring spectrum $E$. We call these DGAs extension DGAs. We also define and study this notion for $E_\infty$ DGAs. 
     
     The topological Hochschild homology (THH) spectrum of an extension DGA splits in a convenient way. We show that formal DGAs with nice homology rings are extension and therefore their THH groups can be  obtained from their Hochschild homology groups in many cases of interest. We also provide interesting examples of DGAs that are not extension.
     
     In the second part, we study properties of extension DGAs. We show that in various cases, topological equivalences and quasi-isomorphisms agree for extension DGAs. From this, we obtain that dg Morita equivalences and Morita equivalences also agree in these cases. 
\end{abstract}
\footnotesize{\textit{\textbf{Keywords:}} Differential graded algebras, Ring spectra, Topological Hochschild homology}

\footnotesize\textit{\textbf{MSC:}}55U99; 55P43; 18G35 
 \normalsize
\section{introduction}

    In \cite{stanley1997dissertation}, Stanley shows that the homotopy category of differential graded algebras is equivalent to the homotopy category of $H\Z$-algebras. Later, Shipley improves this equivalence to a zig-zag of Quillen equivalences between the model categories of DGAs and $H\Z$-algebras \cite{Shipley2007DGAs}. This opens up a new opportunity to study DGAs, i.e.\ to study DGAs using ring spectra. 
    
    Dugger and Shipley use this zig-zag of Quillen equivalences to define new equivalences between DGAs called  topological equivalences, see Definition \ref{def topeq} below. They show non-trivial examples of topologically equivalent DGAs and they use topological equivalences to develop a Morita theory for DGAs \cite{dugger2007topologicalequiv}. In \cite{bayindir2019dgaswithpolynomial}, the author uses topological equivalences to obtain classification results for DGAs. Moreover, topological equivalences for $E_\infty$ DGAs are studied by the author in \cite{bayindir2018topeqeinfty}.
    
    In this work, we follow this philosophy in a different way. We study what we call extension DGAs which are the DGAs that are obtained from ring spectra through the extension of scalars functor from $\Sp$-algebras to $H\Z$-algebras, i.e.\ the functor $H\Z \wedge -$. More generally, we work in $R$-DGAs for a discrete commutative ring $R$. There is a zig-zag of Quillen equivalences between $R$-DGAs and $HR$-algebras \cite{Shipley2007DGAs}. 
    
    \begin{definition} \label{def extension}
    An $R$-DGA $X$ is \textit{$R$-extension} if the $HR$-algebra corresponding to $X$ is weakly equivalent to $HR \wedge E$ for some cofibrant $\Sp$-algebra $E$. For $R=\Z$, we omit $\Z$ and write extension instead of $\Z$-extension. 
     \end{definition}
    To define $R$-extension $E_\infty$ $R$-DGAs, we use the zig-zag of Quillen equivalences between $E_\infty$ $R$-DGAs and commutative $HR$-algebras constructed in \cite{richter2017algebraicmodel}.
    
    \begin{definition} \label{def einfty extension}
     An $E_\infty$ $R$-DGA $X$ is \textit{$R$-extension} if the commutative $HR$-algebra corresponding to $X$ is weakly equivalent to $HR \wedge E$ for some cofibrant commutative $\Sp$-algebra $E$. For $R=\Z$, we omit $\Z$ and write extension instead of $\Z$-extension. 
     \end{definition}
     See Appendix A for a discussion on the compatibility of the two definitions above.
     
    One reason to study extension DGAs is that for an $R$-extension DGA $X$, there is  the following splitting at the level of  spectra
\begin{equation*}    \label{eq thh splitting}
\thh(X) \simeq \thh(HR) \wedge_{HR} \text{HH}^R(X)
     \end{equation*}
     where $\text{HH}^R$ denotes $\thh^{HR}$, see  \cite[Theorem 1]{vogt1998adjoining} for an instance of this splitting when $X$ is the Eilenberg-Mac Lane spectrum of a discrete ring. If $X$ is an $R$-extension $E_\infty$ $R$-DGA, then this splitting is a splitting of commutative ring spectra. This follows from the fact that THH commutes with smash products and the base change formula for THH, see \cite[Conventions]{krause2019b}.
     
     This splitting simplifies THH calculations significantly in many situations. Indeed, it is an important stepping stone in many THH calculations in the literature, particularly for the case where $X$ is a discrete ring, i.e.\ a DGA whose homology is concentrated in degree 0. For example, Larsen and Lindenstrauss show that this splitting exists at the level of homotopy groups for various discrete rings of characteristic $p$ \cite{lindenstrauss2000thhofalgebras}. Furthermore, Hesselholt and Madsen prove such a splitting for discrete rings that have a nice basis with respect to the ground ring $R$  \cite[Theorem 7.1]{hesseholt1997ktheoryofwitt}. In the following theorem,  we generalize this result to connective formal DGAs. Note that a connective DGA is a DGA whose negative homology is trivial. 
     
\begin{theorem}\label{thm formal are extension}
Let $X$ be a connective formal $R$-DGA whose homology has a homogeneous basis as an $R$-module containing the multiplicative unit such that the multiplication of two basis elements is either zero or a basis element. In this situation, $X$ is $R$-extension. As a result, we have the following equivalence of spectra.
\[\thh(X) \simeq \thh(HR) \wedge_{HR} \textup{HH}^R(X)\]
\end{theorem}
Section \ref{sec formaldgas to hzalgebras} is devoted to the proof of this theorem. Furthermore, for a given $R$-DGA that satisfies the hypothesis of the theorem above, we provide an explicit description of the corresponding $HR$-algebra; see Proposition \ref{prop formal hzalg}. The author and Moulinos show that for such $HR$-algebras, one often obtains non-trivial splittings at the level of topological negative cyclic homology and topological periodic homology  \cite[4.8 and 6.1]{bayindir2020algkthryofTHHFp}. Using these splittings, the author and Moulinos compute the algebraic K-theory of $\text{THH}(H\F_p)$, i.e.\ the algebraic K-theory of  the formal DGA with homology $\fp[x_2]$. In a future work, the author is planning to compute the algebraic K-theory groups of various formal DGAs by using  Proposition \ref{prop formal hzalg} and the splittings provided in \cite{bayindir2020algkthryofTHHFp}. 
\begin{remark}\label{rem on graded monoid rings}
Another way to state the hypothesis of  Theorem \ref{thm formal are extension} is the following. Let $M$ be a monoid in the category of  graded pointed sets. From $M$, one obtains a graded $R$-algebra $R \langle M \rangle$ whose underlying $R$-module is the free $R$-module over the graded set $M_{-}$ obtained by removing the based point from $M$. The multiplication on $R \langle M \rangle$ is given by the multiplication on $M$ where the based point of $M$ is considered as the zero element in $R \langle M \rangle$. A graded $R$-algebra of the form $R \langle M \rangle$ is called a graded monoid $R$-algebra. With this definition, a connective formal $R$-DGA satisfies the hypothesis of Theorem \ref{thm formal are extension} if and only if its homology is a graded monoid $R$-algebra.
\end{remark}

\begin{remark}
We mention a few examples of graded rings that satisfy the hypothesis of the  theorem above as homology of $X$. The polynomial algebra over $R$ with a non-negatively graded set $S$ of generators $R[S]$ satisfies the hypothesis if all the elements of $S$ are in even degrees. The basis of $R[S]$ is given by the monomials in $S$ and the unit $1 \in R$. Similarly, many examples of quotients of  polynomial rings with even degree generators  also satisfy this hypothesis; for example $R[x]/(x^2)$, $R[x,y]/(y^2)$ and $R[x,y]/(x^2y,y^3)$ with even $\lv x \rv$ and $\lv y \rv$. However, there are rings that do not satisfy this hypothesis. For example for $R= \Z$, the exterior algebra on two generators $\Lambda[x,y] \cong \Lambda[x] \otimes \Lambda[y]$ with odd $\lvert x \rvert$ and $\lvert y \rvert$ has a basis given by $\{ x, y, xy \}$, but $yx = -xy$ and therefore $yx$ is not one of the basis elements. Indeed, $\Lambda[x,y]$ has no basis that satisfies this hypothesis.
\end{remark}

    We prove the following non-extension results. 

\begin{theorem}(Theorem \ref{thm 1 restated}) \label{thm 1}
Let $Y$ be an $E_\infty$ DGA. If $Y$ is quasi-isomorphic to an $E_\infty$ $\F_p$-DGA then $Y$ is not an extension $E_\infty$ DGA.
\end{theorem}

\begin{theorem}(Theorem \ref{thm nonext for p is 2 restated})\label{thm nonext for p is 2} Let $X$ be a DGA. If $X$ is quasi-isomorphic to an $\F_2$-DGA then $X$ is not an extension DGA.

\end{theorem}
\begin{remark}
These theorems should be compared with the two commutative $H\Z$-algebras $X$ and $Y$ obtained from $H\Z \wdg \hfp$ through the structure maps 
\[\hz \cong \hz \wdg \sph  \to \hz \wdg \hfp\]
and 
\[\hz \cong \sph \wdg \hz \to \sph \wdg \hfp \to \hz \wdg \hfp\]
respectively. The $E_\infty$ DGA corresponding to $X$ is an extension $E_\infty$ DGA and the $E_\infty$ DGA corresponding to $Y$ is an $E_\infty$ $\fp$-DGA. Although these two $E_\infty$ DGAs are $E_\infty$ topologically equivalent, they are not quasi-isomorphic due to Theorem 5.3 in \cite{bayindir2018topeqeinfty}.  For the associative case with $p=2$, the distinction between the two DGAs corresponding to $X$ and $Y$ is due to Example 5.6 in \cite{dugger2007topologicalequiv}.
\end{remark}

In the results above, we work with ($E_\infty$) DGAs in mixed characteristic, i.e.\ we work in ($E_\infty$) $\Z$-DGAs. A natural question to ask is if there are examples of $E_\infty$ $k$-DGAs that are not $k$-extension for a field $k$. In Example \ref{example of non extension} below, we show that there are $E_\infty$ $\F_p$-DGAs that are not $\F_p$-extension.

Now we discuss topological equivalences of DGAs and the properties of extension DGAs regarding topological equivalences.

\begin{definition} \label{def topeq}
Two DGAs $X$ and $Y$ are \textit{topologically equivalent} if the corresponding $H\Z$-algebras $HX$ and $HY$ are weakly equivalent as $\Sp$-algebras. 
\end{definition}

The definition of $E_\infty$ topological equivalences is as follows.

\begin{definition} \label{def einfty topeq}
Two $E_\infty$ DGAs $X$ and $Y$ are \textit{$E_\infty$ topologically equivalent} if the corresponding \textit{commutative} $H\Z$-algebras $HX$ and $HY$ are weakly equivalent as \textit{commutative} $\Sp$-algebras. 
\end{definition}

It follows from these definitions that quasi-isomorphic ($E_\infty$) DGAs are ($E_\infty$) topologically equivalent. However, there are examples of non-trivially topologically equivalent DGAs, i.e.\ DGAs that are topologically equivalent but not quasi-isomorphic \cite{dugger2007topologicalequiv}. 
Furthermore, examples of non-trivially $E_\infty$ topologically equivalent $E_\infty$ DGAs are constructed by the author in \cite{bayindir2018topeqeinfty}. 

\begin{example} \label{example of non extension}
This is an example of $E_\infty$ $\F_p$-DGAs that are not $\F_p$-extension. In Example 5.1 of \cite{bayindir2018topeqeinfty}, the author constructs  non-trivially $E_\infty$ topologically equivalent $E_\infty$ $\F_p$-DGAs that we call $X$ and $Y$, i.e.\ $X$ and $Y$ are $E_\infty$ topologically equivalent but they are not quasi-isomorphic. Although  these $E_\infty$ $\F_p$-DGAs are $E_\infty$ topologically equivalent, their Dyer-Lashof operations are different. 

For $p=2$, the homology rings of these $E_\infty$ $\F_p$-DGAs  are given by  
\[\F_2[x]/(x^4)\]
for both $X$ and $Y$ where $ \lvert x \rvert = 1$. On the homology of $X$, the first Dyer--Lashof operation is trivial, i.e.\ $\mathrm{Q}^1x=0$. On the other hand, we have  $\mathrm{Q}^1 x = x^3$ on the homology of $Y$. Using these properties we show (for all primes) that these $E_\infty$ $\F_p$-DGAs are not $\F_p$-extension $E_\infty$ $\fp$-DGAs. See Section \ref{proof of example of nonextension} for a proof of this fact. 
\end{example}

By Theorem 1.6 in \cite{bayindir2018topeqeinfty}, $E_\infty$ topological equivalences between $E_\infty$ $\F_p$-DGAs with trivial first homology preserve Dyer-Lashof operations. We prove a stronger result for  $\F_p$-extension  $E_\infty$ $\F_p$-DGAs.

\begin{theorem}\label{thm einfty te eq qiso}
Let $X$ be an $\F_p$-extension $E_\infty$ $\F_p$-DGA with $H_1X = 0$ and let $Y$ be an $E_\infty$ $\F_p$-DGA.  In this situation, $X$ and $Y$ are quasi-isomorphic if and only if they are $E_\infty$ topologically equivalent.
\end{theorem}

In the following results, we show various situations where  topological equivalences and quasi-isomorphisms agree. 

\begin{theorem} \label{thm assoc te eq qiso}
Let ($p=2$) $p$ be an odd prime. Let $X$ be an  extension $\F_p$-DGA whose homology is trivial on degrees ($2^r-1$)  $2p^r-2$ for $r \geq 1$ and () $2p^s-1$ for $s \geq 0$ and let $Y$ be an $\F_p$-DGA. In this situation, $X$ and $Y$ are quasi-isomorphic if and only if they are topologically equivalent.
\end{theorem}

For the corollary below, note that  a co-connective DGA is a DGA with trivial homology in positive degrees.

\begin{corollary}
Let $X$ be a co-connective extension  $\F_p$-DGA and let $Y$ be an $\F_p$-DGA. Then $X$ and $Y$ are quasi-isomorphic if and only if they are topologically equivalent.
\end{corollary}

For the theorem and the corollary below, let $R= \Z/(m)$ for some integer $m \neq \pm 1$.
\begin{theorem}\label{thm formal topological equivalences}
Let $X$ be an $R$-DGA whose corresponding $HR$-algebra is equivalent to $HR \wdg Z$ for some cofibrant $\sph$-algebra $Z$ whose underlying spectrum is equivalent to a coproduct of suspensions and/or desuspensions of the sphere spectrum. Also, let $Y$ be an $R$-DGA. Then $X$ and $Y$ are quasi-isomorphic if and only if they are topologically equivalent.
\end{theorem}

Our main interest for this theorem is due to its corollary stated below. This follows by Proposition \ref{prop formal hzalg} which implies that an $R$-DGA that satisfies the hypothesis of Theorem \ref{thm formal are extension} also satisfies the hypothesis of the theorem above.

\begin{corollary}
\label{prop formal topological equivalences}
Let $Y$ be an $R$-DGA and let $X$ be as in Theorem \ref{thm formal are extension}. Then $X$ and $Y$ are quasi-isomorphic if and only if they are topologically equivalent.
\end{corollary}

Two DGAs $X$ and $Y$ are said to be \textit{Morita equivalent} if the model categories of $X$-modules and $Y$-modules are Quillen equivalent. There is a stronger notion of Morita equivalence for DGAs called  \textit{dg Morita equivalences} defined by Keller, see \cite[Section 3.8]{keller2006ondifferential} and  \cite[7.6]{dugger2007topologicalequiv}. This is a strictly stronger notion of Morita equivalence since there are examples of DGAs that are Morita equivalent but not dg Morita equivalent \cite[Section 8]{dugger2007topologicalequiv}. However in the situations where topological equivalences and quasi-isomorphisms agree, these two notions of Morita equivalences also agree, see Proposition 7.7 and Theorem 7.2 of \cite{dugger2007topologicalequiv}. We obtain the following corollary to Theorems \ref{thm formal topological equivalences} and \ref{thm assoc te eq qiso}.
\begin{corollary}
Assume that $X$ and $Y$ are as in Theorem \ref{thm formal topological equivalences} or as in Theorem \ref{thm assoc te eq qiso}. Then $X$ and $Y$ are Morita equivalent if and only if they are dg Morita equivalent. 
\end{corollary}

\textbf{Organization} In Section \ref{sec dualsteenrodalgebra}, we describe the dual Steenrod algebra and the Dyer--Lashof operations on it. In Section \ref{sec proofofmany}, we prove Theorems \ref{thm einfty te eq qiso}, \ref{thm assoc te eq qiso}  and \ref{thm formal topological equivalences}.  Section \ref{sec thms nonextension} is devoted to the proof of Theorems \ref{thm 1} and \ref{thm nonext for p is 2}. In section \ref{sec formaldgas to hzalgebras}, we prove Theorem \ref{thm formal are extension}. This section is independent from Sections \ref{sec dualsteenrodalgebra}, \ref{sec proofofmany} and \ref{sec thms nonextension} and it contains  explicit descriptions of the $H\Z$-algebras corresponding to the formal DGAs as in Theorem \ref{thm formal are extension} which is of independent interest. We left the proof of Theorem \ref{thm formal are extension} to the end because it uses different tools than the rest of the proof in this work. Appendix A is devoted to a discussion on the compatibility of Definitions \ref{def extension} and \ref{def einfty extension}.

\textbf{Terminology} We work in the setting of symmetric spectra in simplicial sets \cite{hovey2000symmetricspectra}. For commutative ring spectra, we use the positive $\sph$-model structure developed in  \cite{Shipley2004convenient}. When we work in the setting of associative ring spectra, we use the stable model structure of \cite{hovey2000symmetricspectra}. Throughout this work, $R$ denotes a   general discrete commutative  ring except in Section \ref{sec proof of topological equivalence of formal dgas} where $R$ denotes a quotient of $\Z$. When we say ($E_\infty$) DGA, we mean ($E_\infty$) $\Z$-DGA. 

\textbf{Acknowledgements} The author would like to thank Don Stanley for suggesting
to study extension DGAs and also for showing  the construction of the monoid object in  Construction \ref{constr for graded monoids}. I also would like to thank Dimitar Kodjabachev and Tasos Moulinos for a careful reading of this work.

\section{The dual Steenrod algebra} \label{sec dualsteenrodalgebra}

Here, we recall the ring structure and the Dyer--Lashof operations on the dual Steenrod algebra. Using the standard notation, we denote the dual Steenrod algebra by $\mathcal{A}_*$. We have $\pi_*(H\F_p \wedge H\F_p) \cong \mathcal{A}_*$. Milnor shows that the dual Steenrod algebra is a free graded commutative $\fp$-algebra \cite{milnor1958steenrod}. 

For $p=2$, $\mathcal{A}_*$ is given by
\[ \mathcal{A}_* = \F_2[\xi_r \ \lvert \ 
 r\geq 1] = \F_2[\zeta_r \ \lvert \ r \geq 1] \]

\noindent where $\lvert \xi_r \rvert = \lvert \zeta_r \rvert =  2^r-1$. Let $\chi$ denote the action of the transpose map of the smash product on $\pi_*(H\F_p \wedge H\F_p$). We have $\chi(\xi_r)=\zeta_r$.

For an odd prime $p$, the following describes the dual Steenrod algebra. 
\[ \mathcal{A}_* = \F_p[\xi_r \ \lvert \ 
 r\geq 1] \otimes \Lambda(\tau_s \ \lvert \  s \geq 0) = \F_p[\zeta_r \ \lvert \ r \geq 1] \otimes \Lambda(\overline{\tau}_s \ \lvert \  s \geq 0) \]
\noindent Where $\lvert \xi_r \rvert = \lvert \zeta_r \rvert =  2(p^r-1)$ and $\lvert \tau_s \rvert = \lvert \overline{\tau}_s \rvert =  2p^s-1$. In this case, we have $ \chi(\xi_r) =\zeta_r$ and $\chi(\tau_r)=\overline{\tau}_r $.

 Dyer-Lashof operations are power operations that act on the homotopy ring of $H_\infty$ $H\F_p$-algebras \cite{brunerh}. By forgetting structure, commutative $H\F_p$-algebras are examples of $H_\infty$ $H\F_p$-algebras and therefore Dyer--Lashof operations are also defined on the homotopy ring of commutative $H\F_p$-algebras and maps of commutative $H\F_p$-algebras preserve these operations. 
 For $p=2$, there is a Dyer--Lashof operation denote by $\mathrm{Q}^s$ for ever integer $s$ where $\mathrm{Q}^s$ increases the degree by $s$. For odd $p$, there are Dyer--Lashof operations denoted by $\beta \mathrm{Q}^s$ and $\mathrm{Q}^s$ for every integer $s$ that increase the degree by $2s(p-1)-1$ and $2s(p-1)$ respectively. See \cite[\rom{3}.1.1]{brunerh} for further properties of these operations.

 With the unit map 
\[\hfp \cong \hfp \wdg \sph \to \hfp \wdg \hfp,\] 
$H\F_p \wedge H\F_p$ is a commutative $H\F_p$-algebra and therefore Dyer--Lashof operations are defined on the dual Steenrod algebra. These operations are first studied in \cite[\rom{3}.2]{brunerh}. Steinberger shows that the degree one element $\tau_0$ for odd $p$ and $\xi_1$ for $p=2$ generates the dual Steenrod algebra as an algebra with Dyer--Lashof operations, i.e.\ as an algebra over the Dyer--Lashof algebra. In particular for $p=2$, we have 
 \[ \mathrm{Q}^{2^s-2} \xi_1 = \zeta_s \ \ \text{for} \ s > 1. \] 
 For odd $p$, we have 
\begin{equation*} 
\begin{split}
\mathrm{Q}^{(p^s-1)/(p-1)} \tau_0 &= (-1)^s \overline{\tau}_s \\
\beta \mathrm{Q}^{(p^s-1)/(p-1)} \tau_0 &= (-1)^s \zeta_s. 
\end{split}
\end{equation*}
 for $s \geq 1$.
 
\section{Proof of the results on topological equivalences and the non-extension example} \label{sec proofofmany}

In this section, we prove Theorems  \ref{thm einfty te eq qiso}, \ref{thm assoc te eq qiso}  and \ref{thm formal topological equivalences} which provide comparison results on ($E_\infty$) topological equivalences and quasi-isomorphisms of ($E_\infty$) DGAs for various cases. At the end, we prove Proposition \ref{prop for the  example} which justifies the last claim in Example \ref{example of non extension}. This provides examples of $E_\infty$ $\F_p$-DGAs that are not $\F_p$-extension. 

These results are obtained using  similar arguments. Therefore, we suggest the reader to go through their proof in the order presented in this section. 

\subsection{Proof of Theorem \ref{thm einfty te eq qiso} and Theorem \ref{thm assoc te eq qiso}}
The proof of Theorem \ref{thm einfty te eq qiso} and Theorem \ref{thm assoc te eq qiso} are similar. Therefore, we combine them in a single proof. 

In the proof of Theorems \ref{thm einfty te eq qiso} and \ref{thm assoc te eq qiso} and also in the proof of Theorem \ref{thm formal topological equivalences} and Proposition \ref{prop for the  example}, we show that for various $R$-extension ($E_\infty$) $R$-DGAs, ($E_\infty$) topological equivalences and    quasi-isomorphisms agree. 

For this, we use the same technique to produce a quasi-isomorphism, i.e.\ an $HR$-algebra equivalence, out of a given topological equivalence, i.e.\ an $\sph$-algebra equivalence. We start by describing this technique.

Let us focus on the $E_\infty$ case. Assume that we are given commutative $HR$-algebras $Y$ and $HR \wdg Z$ where $Z$ denotes a cofibrant commutative $\sph$-algebra and assume that we are given a weak equivalence 
\[\varphi \co HR \wdg Z \we Y\]
of commutative $\sph$-algebras. Using $\varphi$, we produce a map of commutative $HR$-algebras  through the following composite.  
\begin{equation*} 
 \begin{tikzcd}
 \psi \co HR \wedge Z  \cong HR \wedge \Sp \wedge Z \arrow[r,"i"] &[-8pt] HR \wedge HR \wedge Z  \arrow[r,swap,"\simeq"] \arrow[r,"HR \wedge \varphi"]
 &[4pt]
 HR \wedge Y \ar[r,"m"]  
 &  Y
 \end{tikzcd}
 \end{equation*}
Here, $i$ is the canonical map induced by the unit map $\sph \to HR$ of $HR$ and $m$ is the commutative $HR$-algebra structure map of $Y$. Except $Y$, we provide the objects in the composite above with the commutative $HR$-algebra structure coming from the first $HR$ factor. The maps $i$ and $HR \wdg \varphi$ are maps of commutative $HR$-algebras as they are obtained using the functor $HR \wdg - $ from the category of commutative $\sph$-algebras to the category of commutative $HR$-algebras. Note that  $m$ is the left adjoint of the identity map of $Y$ under the adjunction between the categories of commutative $\sph$-algebras and  commutative $HR$-algebras whose left adjoint is given by the extension of scalars functor $HR \wdg - $ and whose right adjoint is given by the restriction of scalars functor. In particular, this shows that $m$ is also a map of commutative $HR$-algebras. We deduce that $\psi$ is a map of commutative $HR$-algebras as it is given by a composite of such maps. Compared to the commutative case, the definition of the map $\psi$ is slightly more complicated in the associative case as we consider various cofibrant replacements. The results we prove in this section are obtained by showing that $\psi$ is an equivalence under the given hypothesis.

In the proof below, we denote the category of commutative $E$-algebras by $E\calg$ and the category of associative $E$-algebras by $E \alg$ for a given commutative ring spectrum $E$.
\begin{proof}[Proof of Theorem \ref{thm einfty te eq qiso} and Theorem \ref{thm assoc te eq qiso}]
First, we prove Theorem  \ref{thm einfty te eq qiso}. After that, we show how this proof should be modified to obtain Theorem \ref{thm assoc te eq qiso}.

Since quasi-isomorphic $E_\infty$ DGAs are always $E_\infty$ topologically equivalent, we only need to show that if  $X$ and $Y$ are $E_\infty$ topologically equivalent then they are quasi-isomorphic as $E_\infty$ $\F_p$-DGAs.

Let $H\F_p$ denote a cofibrant model of $H\F_p$ in $\Sp \calg$. The category of commutative $\hfp$-algbera spectra is the same as the category of commutative $\sph$-algebra spectra under $\hfp$. Therefore we have a model structure on $H\F_p \calg$ where the cofibrations, fibrations and weak equivalences are precisely the maps that forget to cofibrations, fibrations and weak equivalences in $\Sp \calg$. 
We let $Y$ also denote the commutative $H\F_p$-algebra corresponding to the $E_\infty$ DGA $Y$. Therefore $\pi_1(Y) = 0$. Taking a fibrant replacement, we assume $Y$ is fibrant both in $H\F_p \calg$ and in $\Sp \calg$. Furthermore,  we let $H\F_p \wedge Z$ denote the commutative $H\F_p$-algebra corresponding to the  extension $E_\infty$ $\F_p$-DGA $X$ where $Z$ is a cofibrant object in $\Sp \calg$. This ensures that $H\F_p \wedge Z$ is cofibrant in $H\F_p \calg$. Therefore the composite $\Sp \cof H\F_p \cof H\F_p \wedge Z$ is also a cofibration in $\Sp \calg$; this shows that $H\F_p \wedge Z$ is also cofibrant in $\Sp \calg$. To prove Theorem \ref{thm einfty te eq qiso}, we need to show that $H\F_p \wedge Z$ and $Y$ are weakly equivalent in $H\F_p\calg$.


Because $H\F_p \wedge Z$ and $Y$ are obtained from $E_\infty$ topologically equivalent $E_\infty$ DGAs, they are equivalent as commutative $\Sp$-algebras. Furthermore $H\F_p \wedge Z$ is cofibrant and $Y$ is fibrant, therefore there is a weak equivalence $\varphi \co H\F_p \wedge Z \we Y$
of commutative $\Sp$-algebras. We consider the composite map 
\begin{equation} \label{diag we in HFp calg}
 \begin{tikzcd}
 \psi \co H\F_p \wedge Z  \cong H\F_p \wedge \Sp \wedge Z \arrow[r,"i"] &[-8pt] H\F_p \wedge H\F_p \wedge Z  \arrow[r,swap,"\simeq"] \arrow[r,"H\F_p \wedge \varphi"]
 &[4pt]
 H\F_p \wedge Y \ar[r,"m"]  
 &  Y
 \end{tikzcd}
 \end{equation}
 where  the first map is induced by the unit map $u_{\hfp}\co \Sp \to H\F_p$ of $H\F_p$ and the last map is the $H\F_p$ structure map of $Y$. If we consider all the objects in this composite except $Y$ to have the $H\F_p$ structure coming from the first smash factor, then all objects involved are commutative $H\F_p$-algebras and the maps involved are maps of commutative $H\F_p$-algebras. Note that $i$ and $\hfp \wdg \varphi$ are maps of commutative $\hfp$-algebras as they are obtained via the functor $\hfp \wdg -  \co \sph \calg \to \hfp \calg$. The last map $m$ is a map of commutative $\hfp$-algebras because it is the left adjoint of the identity map of $Y$ under the usual adjunction between $\sph\calg $ and $\hfp \calg$. Since all the maps in the composite above are maps of commutative $\hfp$-algebras, we deduce that $\psi$ is a map of commutative $\hfp$-algebras.
 
 What remains is to show that $\psi$ is a weak equivalence. For this, we take the homotopy groups of the composite defining $\psi$ and show that it is an isomorphism. Firstly, we have a splitting \[H\F_p \wedge H\F_p \wedge Z \cong (H\F_p \wedge H\F_p )\wedge_{H\F_p} (H\F_p \wedge Z)\] in $H\F_p \calg$ where we consider the object on the right hand side of the equality with the $H\F_p$ structure given by the first smash factor instead of the canonical one given by the smash product $\wedge_{H\F_p}$. Because the  homotopy of $H\F_p$ is a field, we have $\pi_*(H\F_p \wedge H\F_p \wedge Z) \cong \mathcal{A}_* \otimes \pi_*(H\F_p \wedge Z)$, see  \cite[\rom{4}.4.1]{ekmm}. With this identification, we obtain that the composite map induced in homotopy by the composite defining $\psi$ is given by 
 \begin{equation} \label{psi star diagram}
 \begin{tikzcd}
 \psi_* \co \pi_*(H\F_p \wedge Z)   \arrow[r,"i_*"] &  \mathcal{A}_* \otimes \pi_*(H\F_p \wedge Z) \arrow[rr,swap,"\cong"] \arrow[rr,"\pi_*(H\F_p \wedge \varphi)"]
 &&
 {H\F_p}_* Y \arrow[r,"m_*"]  
 &  Y_*.
 \end{tikzcd}
 \end{equation}

Note that although we identify the domain of  $\pi_*(\hfp \wdg \varphi)$ as a tensor product, we do not claim that $\pi_*(\hfp \wdg \varphi)$ splits as a tensor product of two maps. 

Now we state and prove the following claims. Afterwards, we combine them to prove that $\psi_*$ is an isomorphism by showing $\psi_*= \varphi_*$.
 
 \underline{Claim 1:} The composite $m_* \circ \pi_*(H\F_p \wedge \varphi)$ maps every element of the form $a\otimes x$ with $\lvert a \rvert > 0$ to zero in $Y_*$.
 
 We have a canonical  map \[(H\F_p \wedge H\F_p) \wedge_{H\F_p} H\F_p \to (HF_p \wedge H\F_p )\wedge_{H\F_p} (H\F_p \wedge Z).\] This map is in $H\F_p \calg$ therefore the induced map in homotopy preserves the Dyer--Lashof operations. The induced map in homotopy is given by the inclusion $\mathcal{A}_* \otimes \F_p \to \mathcal{A}_* \otimes \pi_*(H\F_p \wedge Z)$ and this shows that Dyer-Lashof operations on this subset of  $\mathcal{A}_* \otimes \pi_*(H\F_p \wedge Z)$ are given by the action of the Dyer--Lashof operations on the dual Steenrod algebra i.e.\ $\mathrm{Q}^s (a \otimes 1)= (\mathrm{Q}^s a) \otimes 1$. Let $p$ be an odd prime. Since $\pi_1 (Y)$ is trivial, $m_* \circ \pi_*(H\F_p \wedge \varphi)(\tau_0 \otimes 1) = 0$. Because the dual Steenrod algebra is generated with the Dyer--Lashof operations by $\tau_0$, this shows that $m_* \circ \pi_*(H\F_p \wedge \varphi) (a \otimes 1)= 0 $ for all $a \in \mathcal{A}_*$ with $\lvert a \rvert >0$. Since all maps involved are ring maps and $a \otimes x = (a\otimes 1)(1 \otimes x)$, this finishes the proof of our claim. Note that for $p=2$, one uses $\xi_1$ instead of $\tau_0$. 
 
 \underline{Claim 2:} We have $m_* \circ \pi_*(H\F_p \wedge \varphi)(1 \otimes x) = \varphi_*(x)$ for every $x \in \pi_*(H\F_p \wedge Z)$. 
 We consider the following commutative diagram.
 \begin{equation} \label{diag claim 2}
 \begin{tikzcd}
  (\Sp \wedge H\F_p) \wedge_{H\F_p} (H\F_p \wedge Z) \arrow[r,"\cong"] \arrow[d,"h"] & \Sp \wedge H\F_p \wedge Z \arrow[r,"\Sp \wedge \varphi"] \ar[d] & \Sp \wedge Y \ar[d,"h_Y"] \\
 (H\F_p \wedge H\F_p)\wedge_{H\F_p} (H\F_p \wedge Z) \ar[r,"\cong"] & H\F_p \wedge H\F_p \wedge Z \ar[r,"H\F_p \wedge \varphi"]   & H\F_p \wedge Y  \arrow[r,"m"] & Y
 \end{tikzcd}
 \end{equation}
 Because $Y$ is in $H\F_p \calg$, $m \circ h_Y = id$. We have $m_* \circ \pi_*(H\F_p \wedge \varphi)(1 \otimes x) = m_* \circ \pi_*(H\F_p \wedge \varphi) \circ h_*(x)$. Carrying $x$ through the top row and then composing with $m \circ h_Y$, we obtain the equality in our claim.
 
 \underline{Claim 3:} We have $i_* (x) = 1 \otimes x + \Sigma_{i} a_i \otimes x_i$ for some $a_i \in \mathcal{A}_*$ with $\lvert a_i\rvert >0$ and $x_i \in \pi_*(H\F_p \wedge Z)$.
 
The composite of the maps below is the identity
\[H\F_p \wedge Z \cong H\F_p \wedge \Sp \wedge Z \xrightarrow{i} H\F_p \wedge H\F_p \wedge Z \xrightarrow{m_{\hfp} \wdg id} H\F_p \wedge Z \]
where $m_{\hfp}$ is the multiplication map of $H\F_p$. With the identification 
\[H\F_p \wedge H\F_p \wedge Z \cong (H\F_p \wedge H\F_p )\wedge_{H\F_p} (H\F_p \wedge Z),\]
we obtain the following composite in homotopy 
\begin{equation} \label{diag retract}
\pi_*(H\F_p \wedge Z) \xrightarrow{i_*} \mathcal{A}_* \otimes \pi_*(H\F_p \wedge Z) \xrightarrow{\pi_*(m_{\hfp} \wdg id)} \pi_*(H\F_p \wedge Z)
\end{equation}
where $\pi_*(m_{\hfp} \wdg id)$ is given by the augmentation $\mathcal{A}_* \to \F_p$. This description of $\pi_*(m_{\hfp} \wdg id)$ and the fact that $\pi_*(m_{\hfp} \wdg id) \circ i_*= id$ proves our claim.

Finally, we have 
\begin{equation*}
    \begin{split}
        \psi_*(x) & = m_* \circ \pi_*(H\F_p \wedge \varphi) \circ i_* (x) \\
                  & = m_* \circ \pi_*(H\F_p \wedge \varphi) (1 \otimes x + \Sigma_{i} a_i \otimes x_i) \\
                  & = \varphi_*(x)
    \end{split}
\end{equation*}
 for some $a_i \in \mathcal{A}_*$ with $\lvert a_i \rvert >0$. Here, the first equality follows by the definition of $\psi_*$, the second equality follows by Claim 3 and the third follows by Claim 2 and Claim 1. This proves that $\psi_*$ is an isomorphism and therefore $\psi$ is a weak equivalence. At this point, we are done with the proof of Theorem \ref{thm einfty te eq qiso}.

 \bigskip
 
Note that for Theorem \ref{thm assoc te eq qiso}, we work in the setting of associative algebras. In this case, we need to be more careful with cofibrant replacements since the forgetful functor from $H\F_p \alg$ to $\Sp \alg$ does not necessarily preserve cofibrant objects. Let $H\F_p$ be as before and let $Z$ be cofibrant in $\Sp \alg$ such that $H\F_p \wedge Z$ is an $H\F_p$-algebra that corresponds to $X$. By abuse of notation, let $Y$ be a fibrant $H\F_p$-algebra corresponding to $Y$. Let $T \trfib H\F_p \wedge Z$ be a cofibrant replacement of $H\F_p \wedge Z$ in $\Sp \alg$.  We have the following lift
\begin{equation} \label{diag lifting to the replacement}
    \begin{tikzcd}
     \Sp \arrow[r] \arrow[d,rightarrowtail] & T \arrow[d,twoheadrightarrow,"\sim"]
    \\
     Z \ar[r] \arrow[ur,dashed,"f"]&  H\F_p \wedge Z
    \end{tikzcd}
\end{equation} 
in $\Sp \alg$ where the bottom map is given by the map $Z \cong \Sp \wedge Z \to H\F_p \wedge Z$. Since $T$ and $Y$ are obtained from topologically equivalent DGAs, they are equivalent in $\Sp \alg$. Also because $T$ is  cofibrant and $Y$ is fibrant, we have a weak equivalence $\varphi \co T \we Y$ of $\Sp$-algebras. We obtain the following composite map of $H\F_p$-algebras 
\begin{equation*}
 \begin{tikzcd}
 \psi \co H\F_p \wedge Z   \arrow[r,"i"] &[-8pt] H\F_p \wedge T \arrow[r,swap,"\simeq"] \arrow[r,"H\F_p \wedge \varphi"]
 &[4pt]
 H\F_p \wedge Y \ar[r,"m"]  
 &  Y
 \end{tikzcd}
 \end{equation*}
where $i= H\F_p \wedge f$ and $m$ is the $H\F_p$ structure map of $Y$. The map $m$  is a map of $\hfp$-algebras because it is the left adjoint of the identity map of $Y$ under the usual adjunction between $\hfp \alg$ and $\sph \alg$. Note that we denote $H\F_p \wedge f$ by $i$ because the map $i$ in the composite above should be compared to the map $i$ in \eqref{diag we in HFp calg}.

Again, what remains is to show that $\psi_*$ is an isomorphism. Note that the functor $H\F_p \wedge -$ preserves weak equivalences as described in the proof of Proposition 4.7 in \cite{Shipley2004convenient}. Identifying homotopy groups of $T$ with homotopy groups of $H\F_p \wedge Z$ through the trivial fibration above and similarly identifying the homotopy groups of $H\F_p \wedge T$ with those of $H\F_p \wedge H\F_p \wedge Z$, we obtain a description of $\psi_*$ similar to the one in \eqref{psi star diagram}. 
 \begin{equation*} 
 \begin{tikzcd}
 \psi_* \co \pi_*(H\F_p \wedge Z)   \arrow[r,"i_*"] &  \mathcal{A}_* \otimes \pi_*(H\F_p \wedge Z) \arrow[rr,swap,"\cong"] \arrow[rr,"\pi_*(H\F_p \wedge \varphi)"]
 &&
 {H\F_p}_* Y \arrow[r,"m_*"]  
 &  Y_*
 \end{tikzcd}
 \end{equation*}
It is sufficient to show that the claims above also  hold in this case. Claim 1 follows by the hypothesis that $\pi_* Y$ is trivial at the degrees where the algebra generators of the dual Steenrod algebra are. Claim 2 follows similarly. For Claim 3, consider the following sequence of maps
\[H\F_p \wedge Z \xrightarrow{i} H\F_p \wedge T \we H\F_p \wedge H\F_p \wedge Z \xrightarrow{m_{H\fp} \wdg id} H\F_p \wedge Z\]
where $m_{H\fp}$ is the multiplication map of $H\fp$. Due to Diagram \eqref{diag lifting to the replacement}, the composite above is the identity map. Taking homotopy groups of the composite above and omitting the equivalence in the middle, one obtains \eqref{diag retract}. The rest of the proof of Claim 3 follows as before.
\end{proof}

\begin{remark}
The proof of Theorem \ref{thm einfty te eq qiso} is showing slightly more. For a given cofibrant $Z$ in $\Sp \calg$ and  a fibrant $Y$ in $H\F_p \calg$ with $\pi_1Y =0$ and an equivalence $H\F_p \wedge Z \we Y$ of $\Sp$-algebras, the map $H\F_p \wedge Z \to Y$ in $H\F_p \calg$ given by the structure map of $Y$ on $H\F_p$ and the map $\Sp \wedge Z \to H\F_p \wedge Z \we Y$ on $Z$ is also a weak equivalence. Note that to construct this map, we use the fact that $H\F_p \wedge Z$ is a coproduct of $H\F_p$ and $Z$ in $\Sp \calg$.
\end{remark}

\subsection{Example \ref{example of non extension}}\label{proof of example of nonextension} Here, we show that the $E_\infty$ $\F_p$-DGAs provided in Example \ref{example of non extension} are not $\F_p$-extension. 
\begin{proposition}\label{prop for the  example}
Let $X$ and $Y$ be as in Example \ref{example of non extension}. As $E_\infty$ $\fp$-DGAs, $X$ and $Y$ are not $\fp$-extension. 
\end{proposition}

\begin{proof}   Recall that in Example \ref{example of non extension}, we provide examples of $E_\infty$ $\fp$-DGAs that are $E_\infty$ topologically equivalent but not quasi-isomorphic. We prove that $X$ is not an extension $E_\infty$ $\F_p$-DGA. In order to show $Y$ is not extension, it suffices to exchange the roles of $X$ and $Y$ in the proof below.

We assume that $X$ is an extension $E_\infty$ $\F_p$-DGA and obtain a contradiction by showing that $X$ and $Y$ are quasi-isomorphic under this assumption. This is similar to the proof of Theorem \ref{thm einfty te eq qiso} that we assume familiarity with.
 Following the constructions there, we obtain a map of commutative $H\F_p$-algebras 
  \begin{equation*} 
 \begin{tikzcd}
 \psi \co H\F_p \wedge Z  \cong H\F_p \wedge \Sp \wedge Z \arrow[r,"i"] &[-8pt] H\F_p \wedge H\F_p \wedge Z  \arrow[r,swap,"\simeq"] \arrow[r,"H\F_p \wedge \varphi"]
 &[4pt]
 H\F_p \wedge Y \ar[r,"m"]  
 &  Y
 \end{tikzcd}
 \end{equation*}
 as in Diagram \eqref{diag we in HFp calg}
 where $H\F_p \wedge Z$ denotes a commutative $H\F_p$-algebra corresponding to $X$ and  $Y$ denotes a commutative $H\F_p$-algebra corresponding to the $E_\infty$ $\F_p$-DGA $Y$ by abusing notation. This is a map of commutative $H\F_p$-algebras as before. Therefore, it is sufficient to show that $\psi_*$ is an isomorphism.
 
 As in \eqref{psi star diagram}, we have the following description of $\psi_*$.
  \begin{equation*}
 \begin{tikzcd}
 \psi_* \co \pi_*(H\F_p \wedge Z)   \arrow[r,"i_*"] &  \mathcal{A}_* \otimes \pi_*(H\F_p \wedge Z) \arrow[rr,swap,"\cong"] \arrow[rr,"\pi_*(H\F_p \wedge \varphi)"]
 &&
 {H\F_p}_* Y \arrow[r,"m_*"]  
 &  Y_*
 \end{tikzcd}
 \end{equation*}
  By Claim 3 in the proof of Theorem \ref{thm assoc te eq qiso}, for every $x \in \pi_*(H\F_p \wedge Z)$ we have 
 \begin{equation} \label{eq claim 3 repeated}
 i_* (x) = 1 \otimes x + \Sigma_{i} a_i \otimes x_i
 \end{equation}
for some $a_i \in \mathcal{A}_*$ with $\lvert a_i\rvert >0$ and $x_i \in \pi_*(H\F_p \wedge Z)$. 
 
 For $p=2$, $\pi_*(H\F_p \wedge Z)\cong \F_2[x]/(x^4)$ with $\lvert x \rvert =1$. By degree reasons, we either have  $i_*(x) = 1 \otimes x$ or $i_*(x) = 1 \otimes x + \xi_1 \otimes 1$. Since $(1 \otimes x + \xi_1 \otimes 1)^4 \neq 0$ but $x^4 = 0$, the second option is not possible. Therefore we have $i_*(x) = 1 \otimes x$. Since $i$ is a map of ring spectra, $i_*$ is multiplicative so we have $i_*(x^l) = 1 \otimes x^l$ for every $l$.   By Claim 2 in the proof of Theorem \ref{thm einfty te eq qiso}, this shows that $\psi_*$ is an isomorphism. This provides a contradiction as $X$ and $Y$ are not quasi-isomorphic as $E_\infty$ $\F_2$-DGAs.  
 
 For odd $p$, we have \[\pi_*Y \cong \pi_*(H\F_p \wedge Z)\cong \Lambda_{\F_p}[x,y]\] with $\lvert x \rvert =1$, $\lvert y \rvert = 2p-2$. By \eqref{eq claim 3 repeated} above, we have either $i_*(y) = 1 \otimes y$ or $i_*(y) = c \xi_1 \otimes 1 + 1 \otimes y$ for some unit $c \in \F_p$. However, $y^2 = 0$ but $(c\xi_1 \otimes 1 + 1 \otimes y)^2 \neq 0$ so only the first option is possible. This shows that $\psi_*(y)=y$ due to Claim 2 in the proof of Theorem \ref{thm einfty te eq qiso}. The $2p-2$ Postnikov sections of $Y$ and $H\F_p \wedge Z$ agrees with that of $H\F_p \wedge H\F_p$ in commutative $H\F_p$-algebras, see Example 5.1 in \cite{bayindir2018topeqeinfty}. Using this together with the fact that $\beta \mathrm{Q}^1\tau_0 = -\zeta_1$ in the dual Steenrod algebra, we obtain that we have $\beta \mathrm{Q}^1 x = y$ up to a unit both in $\pi_* (H\F_p \wedge Z)$ and in $\pi_*Y$. Because $\psi$ is a map of commutative $H\F_p$-algebras, $\psi_*$ preserves Dyer--Lashof operations. Since $\psi_*(y) =y$, we obtain that $\psi_*(x)=x$ up to a unit of $\F_p$. Because $\psi_*$ is a ring map, we deduce that $\psi_*$ is indeed an isomorphism. Therefore $\psi$ is a weak equivalence of commutative $H\F_p$-algebras between the commutative $H\F_p$-algebras corresponding to the $E_\infty$ $\F_p$-DGAs $X$ and $Y$. This contradicts the fact that $X$ and $Y$ are not quasi-isomorphic as $E_\infty$ $\F_p$-DGAs and finishes our proof.
\end{proof}
\subsection{Proof of Theorem \ref{thm formal topological equivalences}} \label{sec proof of topological equivalence of formal dgas}

In the proof below, $R$ denotes $\Z/(m)$ for some $m \neq \pm 1$. Furthermore, $X$ denotes an $R$-DGA whose corresponding $HR$-algebra is equivalent to $HR \wdg Z$ where $Z$ is a cofibrant $\sph$-algebra whose underlying spectrum is weakly equivalent to a coproduct of suspensions and/or desuspensions of the sphere spectrum. For every $R$-DGA $Y$, we need to show that $X$ and $Y$ are quasi-isomorphic if and only if they are topologically equivalent.
\begin{proof}[Proof of Theorem \ref{thm formal topological equivalences}]
Let $HR$ be cofibrant as a commutative $\sph$-algebra. This guarantees that $HR\wdg -$ preserves weak equivalences, see \cite[proof of Proposition 4.7]{Shipley2004convenient}. Since $HR\wdg -$ preserves weak equivalences, we can further assume $Z$ to be fibrant. 

Let $Y$ be an $R$-DGA. Since quasi-isomorphic $R$-DGAs are always topologically equivalent, we only need to show that $X$ and $Y$ are quasi-isomorphic if they are topologically equivalent. Abusing notation, we also let $Y$ denote a fibrant $HR$-algebra corresponding to the $R$-DGA $Y$. We assume that $X$ and $Y$ are topologically equivalent, i.e.\ $HR \wdg Z$ and $Y$ are equivalent as $\sph$-algebras. Using this, we are going to show that there is a weak equivalence 
\[\psi \co HR \wdg Z \we Y\]
of $HR$-algebras.

Let $g \co T \trfib HR \wdg Z$ be a cofibrant replacement of $HR \wdg Z$ in $\sph$-algebras. As in Diagram \eqref{diag lifting to the replacement}, there exists a map $f \co Z \to T$ such that the following diagram commutes. 
\begin{equation}\label{diag anotherlift}
    \begin{tikzcd}
       \sph \ar[d,rightarrowtail ] \ar[r] & T \ar[d,twoheadrightarrow,"g"]\ar[d,"\simeq",swap]\\
       Z \ar[r,swap,"h_Z"] \ar[ur,dashed,"f"]& HR \wdg Z
    \end{tikzcd}
\end{equation}
Here, $h_Z$ denotes the canonical map 
\[h_Z\co Z \cong \sph \wdg Z \to HR \wdg Z.\]

Since $X$ and $Y$ are topologically equivalent, $T$ and $Y$ are equivalent as $\sph$-algebras. Furthermore, $T$ is cofibrant and $Y$ is fibrant, therefore we have a weak equivalence 
\[\varphi \co T \we Y\]
of $\sph$-algebras.

We obtain the composite map 
\[\psi \co HR \wdg Z \xrightarrow{HR \wdg f} HR \wdg T \xrightarrow{ HR \wdg \varphi} HR \wdg Y \xrightarrow{m} Y\]
of $HR$-algebras where $m$ denotes the $HR$ structure map of $Y$. Note that the last map above is a map of $HR$-algebras as it is the left adjoint of the identity map of $Y$ under the usual adjunction between the categories of $HR$-algebras and $\sph$-algebras.  Since $\psi$ is a map of $HR$-algebras, it is sufficient to show that $\psi$ induces an isomorphism in homotopy. 

We have the following commuting diagram
\begin{equation*}
 \begin{tikzcd}
  \Sp \wedge T \arrow[r,"\Sp \wedge \varphi"] \ar[d,"h_T"] & \Sp \wedge Y \ar[d]\ar[dr,"id"] \\
  HR \wdg T \ar[r,"HR \wedge \varphi"]   & HR \wedge Y  \arrow[r,"m"]  & Y
 \end{tikzcd}
 \end{equation*}
where the vertical maps are the canonical maps induced by the unit map $u_R\co \sph \to HR$. This shows that the composite map starting from $T\cong \sph \wdg T$ and ending in $Y$ is given by $\varphi$ and therefore is a weak equivalence. In particular, $\pi_*(m\circ( HR\wdg \varphi))$ is an isomorphism when it is restricted to the image of the Hurewicz map of $T$
\[\pis h_T \co \pi_*(\sph \wdg T) \to \pi_*(HR \wdg T).\]
Therefore in order to prove that $\psi_*$ is an isomorphism, it is sufficient to show that the map 
\[\pi_*(HR \wdg f)\co \pi_*(HR \wdg Z) \to \pi_*(HR \wdg T)\]
is injective and its image agrees with the image of $\pis h_T$. For this, it is sufficient to  prove that the corresponding statements are true after composing with the isomorphism 
\[\pi_*(HR \wdg g)\co \pi_*(HR \wdg T) \xrightarrow{\cong} \pi_*(HR \wdg HR \wdg Z).\]
In other words, it is sufficient to show that 
\[\pi_*(HR \wdg g) \circ \pi_*(HR \wdg f)\]
is injective and the image of this map agrees with the image of the map $\pi_*(HR \wdg g) \circ \pis h_T$. Due to Diagram \eqref{diag anotherlift}, $g \circ f = h_Z$. Therefore, it is sufficient to show that $\pis(HR \wdg h_Z)$ is injective in homotopy and its image agrees with the image of $\pi_*(HR \wdg g) \circ \pis h_T$.

The following composite is the identity map: 
\[HR \wdg Z \xrightarrow{HR \wdg h_Z} HR \wdg HR \wdg Z \xrightarrow{m \wdg id } HR \wdg Z,\]
where $m$ denotes the multiplication map of $HR$ and $id$ denotes the identity map of $Z$. From this, we deduce that $\pis(HR \wdg h_Z)$ is injective in homotopy as desired. What remains to prove is that the image of $\pis(HR \wdg h_Z)$ agrees with the image of $\pi_*(HR \wdg g) \circ \pis h_T$.

Due to the commuting diagram:
\begin{equation*} 
 \begin{tikzcd}
  \Sp \wedge T \ar[r,"g"]\arrow[r,"\simeq",swap] \ar[d,"h_T"] & \Sp \wedge HR \wdg Z \ar[d,"h_{HR \wdg Z}"]\\
  HR \wdg T  \ar[r,swap,"\simeq"]\ar[r,"HR \wedge g"]   & HR \wedge HR \wdg Z,  
 \end{tikzcd}
 \end{equation*}
 the image of the map $\pi_*(HR \wdg g)\circ \pis h_T$ is given by the image of the Hurewicz map 
 \[ \pi_*(h_{HR \wdg Z})\co \pi_*(\sph \wdg HR \wdg Z)\to \pi_*(HR \wdg HR \wdg Z)\]
 of $HR \wdg Z$. Note that $h_{HR \wdg Z}$ is  induced by the unit map of $HR$ as usual. Therefore, it is sufficient to show that the image of $\pis(HR \wdg h_Z)$ agrees with the image of $ \pi_*(h_{HR \wdg Z})$.

The map $HR \wdg h_Z$ is  the canonical map
\[HR \wedge  Z \cong HR \wedge \Sp \wedge  Z \to HR \wedge HR \wedge  Z.\]
This is the same as the composite \begin{equation}\label{eq twisted hurewicz}
  HR \wedge  Z \cong \Sp \wedge HR \wedge Z \xrightarrow{h_{HR \wdg Z}} HR \wedge HR \wedge  Z 
      \xrightarrow{\tau \wedge id} HR \wedge HR \wedge  Z
\end{equation}
where $\tau$ is the transposition map of the monoidal structure. Since the map $h_{HR \wdg Z}$ in the middle of the composite in \eqref{eq twisted hurewicz} induces $ \pi_*(h_{HR \wdg Z})$, it is sufficient to show that $\pi_*(\tau \wdg id)$ is the identity map on the  image of $\pi_*(h_{HR \wdg Z})$. 

 By hypothesis, the underlying spectrum of $Z$ is a wedge of suspensions of the sphere spectrum. Let 
\[E = \vee_{a \in A} \Sigma^{\lv a \rv}\sph\]
be weakly equivalent to $Z$ as a spectrum where $A$ is a graded set. Since $E$ is cofibrant and $Z$ is fibrant, there is a weak equivalence of spectra $E \we Z$.

This equivalence induces the horizontal maps in the following commuting diagram of $\sph$-modules. 
\begin{equation*} \label{diag ki ne diagram}
 \begin{tikzcd}
 HR \wdg E \arrow[r,"h_{HR \wdg E}"]  \ar[d,"\simeq"]& HR \wedge HR \wdg E   \ar[r,"\tau \wdg id"]\ar[d,"\simeq"] &HR \wdg HR \wdg E\ar[d,"\simeq"]\\
  HR \wedge Z \arrow[r,"h_{HR \wdg Z}"]  & HR \wedge HR \wdg Z  \ar[r,"\tau \wdg id"] &HR \wdg HR \wdg Z 
 \end{tikzcd}
 \end{equation*}
Here, $h_{HR \wdg E}$ denotes the canonical map that induces the Hurewicz map of $HR \wdg E$ in homotopy. In order to show that $\pi_*(\tau \wdg id)$ (of the bottom row) is the identity map on the image of $\pi_*(h_{HR \wdg Z})$, it is sufficient to show that $\pi_*(\tau \wdg id)$ (of the top row) is given by the identity map on the image of $\pi_*(h_{HR \wdg E})$. For this, it is sufficient to show that the composite of the maps on the top row is given by $\pi_*(h_{HR \wdg E})$ in homotopy.

Note that the canonical $R$-module basis elements of \[\pi_*(HR \wdg E) =\pi_* (HR \wedge  (\vee_{a \in A}\Sigma^{\lvert a \rvert} \Sp)) \cong \bigoplus_{a \in A} \Sigma^{\lvert a \rvert} R\]
are also abelian group generators because $R = \Z/(m)$ for some integer $m$. Therefore,  it is sufficient to show that 
\[\pi_*(\tau \wdg id) \circ \pis(h_{HR \wdg E})(x) = \pis(h_{HR \wdg E})(x)\] for every canonical basis element $x$. Such an $x$ is represented by a map \[u_{HR} \wdg i_a \co \Sp \wedge \Sigma^{\lvert a \rvert} \Sp \to  HR \wedge  (\vee_{a \in A}\Sigma^{\lvert a \rvert} \Sp)=HR \wdg E\]  where  $i_a$ is the inclusion of the cofactor corresponding to an $a \in A$.

In other words, it is sufficient to show that the composite
\[\Sp \wedge \Sigma^{\lvert a \rvert}\sph \xrightarrow{u_{HR} \wdg i_a } HR \wdg E \xrightarrow{h_{HR \wdg E}} HR \wdg HR \wdg E \xrightarrow{\tau \wdg id} HR \wdg HR \wdg E\]
agrees with the composite 
\[\Sp \wedge \Sigma^{\lvert a \rvert} \sph\xrightarrow{u_{HR} \wdg i_a} HR \wdg E \xrightarrow{h_{HR \wdg E}} HR \wdg HR \wdg E.\]
To see this, note that the  composite maps above are of the form $\upsilon \wdg i_a$ and  $\nu \wdg i_a$ respectively where $\upsilon$ and $\nu$ are $\sph$-algebra maps from $\sph$ to $HR \wdg HR$. Since $\sph$ is the initial object in the category of $\sph$-algebras, we deduce that $\upsilon = \nu$. Therefore, the two composites above agree as claimed.

\end{proof}

\section{E-infinity $\fp$-DGAs are not extension} \label{sec thms nonextension}

This section is devoted to the proof of  Theorems \ref{thm 1} and \ref{thm nonext for p is 2}. We restate these theorems below. 

\begin{theorem} (Theorem \ref{thm 1}) \label{thm 1 restated}
Let $Y$ be an $E_\infty$ DGA. If $Y$ is quasi-isomorphic to an $E_\infty$ $\F_p$-DGA then $Y$ is not an extension $E_\infty$ DGA.
\end{theorem}

\begin{theorem} (Theorem \ref{thm nonext for p is 2}) \label{thm nonext for p is 2 restated}
Let $X$ be a DGA. If $X$ is quasi-isomorphic to an $\F_2$-DGA then $X$ is not an extension DGA.

\end{theorem}

In the proof of these theorems, we use the ring structure and the Dyer--Lashof operations on $\pi_*(H\F_p \wedge H\Z) = H{\F_p}_* H\Z$. For odd $p$, the ring structure is given by 
\[  H{\F_p}_* H\Z \cong \F_p[\zeta_r \ \lvert \ 
 r\geq 1] \otimes \Lambda(\overline{\tau}_s \ \lvert \  s \geq 1) \]
where the degrees of the generators are the same as those of the dual Steenrod algebra. Note that  $H{\F_p}_* H\Z$ has the same generators as the dual Steenrod algebra except that $H{\F_p}_* H\Z$ does not contain the degree $1$ generator $\tau_0$. Indeed, the map $H{\F_p}_* H\Z \to H{\F_p}_* H\F_p= \mathcal{A}_*$ induced by $H\Z \to H\F_p$ is the canonical inclusion \cite[\rom{2}.10.26]{schwede2007untitled}. This inclusion is induced by a map of commutative $\hfp$-algebras and therefore it preserves the Dyer-Lashof operations. Therefore through this map, the Dyer--Lashof operations on the dual Steenrod algebra determine the Dyer--Lashof operations on $H{\F_p}_* H\Z$, see \cite[\rom{3}.2]{brunerh}. 

For $p=2$, we have 
\[ H{\F_2}_* H\Z = \F_2[\zeta_1^2] \otimes \F_2[\zeta_r \ \lvert \ 
 r\geq 2]\]
where $\lvert \zeta_i \rvert = 2^i-1$ for $i \geq 2$ and $\lvert \zeta_1^2 \rvert = 2$. Again, the canonical map $H{\F_2}_* H\Z \to H{\F_2}_* H\F_2=\mathcal{A}_*$ is the canonical inclusion and this determines the Dyer--Lashof operations on $H{\F_2}_* H\Z$.

 For the rest of this section, we assume that $\hz$ is cofibrant as a commutative $\sph$-algebra and  $\hfp$ is cofibrant as a commutative $\hz$-algebra in the model structure developed in \cite{Shipley2004convenient}. Since the category of commutative $\hz$-algebras is the same as the category of  commutative $\sph$-algebras under $\hz$, cofibrations of commutative  $\hz$-algebras forget to cofibrations of commutative $\sph$-algebras. Therefore, $\hz \to \hfp$ is also a cofibration of commutative $\sph$-algebras. This ensures that $\hfp$ is also cofibrant as a commutative $\sph$-algebra and therefore the functor $\hfp \wdg -$ preserves all weak equivalences, see the proof of Proposition 4.7 in \cite{Shipley2004convenient}.

We start by proving the following lemma. This lemma is obvious if one assumes that for a map of discrete commutative  rings $R \to R'$, the Quillen equivalences of (\cite{richter2017algebraicmodel}) \cite{Shipley2007DGAs}   are compatible with the restriction of scalars functors from ($E_\infty$) $R'$-DGAs to ($E_\infty$) $R$-DGAs and from (commutative) $HR'$-algebras to (commutative) $HR$-algebras. However, there is no such  compatibility result available in the literature and proving it is beyond the scope of this work.
\begin{lemma}\label{lem map hfp to x}
Let $X$ be a ($E_\infty$) DGA that is quasi-isomorphic to an ($E_\infty$) $\fp$-DGA. In this situation, there is a map of (commutative) $\hz$-algebras 
\[c(\hfp) \to HX\]
where $HX$ denotes a fibrant (commutative) $\hz$-algebra corresponding to the ($E_\infty$) DGA $X$. Furthermore, $c(\hfp)$ denotes a cofibrant replacement of  $\hfp$  in (commutative) $\hz$-algebras. In the commutative case, $\hfp$ is cofibrant in commutative $\hz$-algebras due to our standing assumptions and therefore the cofibrant replacement above may be omitted.
\end{lemma}
\begin{proof}
We only prove the $E_\infty$ case; the associative case follows in a similar manner. Assume that  we are using a unital $E_\infty$ operad, i.e.\ an operad given by the monoidal unit $\fp$ in operadic degree zero. Barratt-Eccles operad is an example of a unital $E_\infty$-operad \cite{Berger2004combinatorialoperadactions}. In this situation, $\fp$ is the free $E_\infty$ $\fp$-DGA generated by the trivial $\fp$-chain complex $0$. Therefore, $\fp$ is the initial object in $E_\infty$ $\fp$-DGAs. This, together with the fact that $X$ is quasi-isomorphic to an $E_\infty$ $\fp$-DGA implies that there is a map $\fp\to X$ in the homotopy category  of $E_\infty$ DGAs.

The equivalence of categories between the homotopy categories of commutative  $\hz$-algebras and $E_\infty$ DGAs imply that there is also a map $\hfp \to HX$ in the homotopy category of commutative $\hz$-algebras. Since $HX$ is fibrant in commutative $\hz$-algebras, there is a map $c(\hfp) \to HX$ of commutative $\hz$-algebras as desired.
\end{proof}

The following starts with the proof of Theorem \ref{thm 1} and at the end, we mention how this also shows Theorem \ref{thm nonext for p is 2}.
\begin{proof}[Proof of Theorems \ref{thm 1}  and \ref{thm nonext for p is 2}]

Assume to the contrary that there is an extension $E_\infty$ DGA $X$ that is quasi-isomorphic to an $E_\infty$ $\fp$-DGA.  It follows by Lemma \ref{lem map hfp to x} that there is a map $\hfp \to HX$ of commutative $\hz$-algebras where $HX$ denotes a fibrant commutative $\hz$-algebra corresponding to the $E_\infty$ DGA $X$. In particular, the $\hz$-structure map $\hz \to HX$ of $HX$ factors as 
\[\hz \xrightarrow{\varphi_{\hfp}} \hfp \to HX\]
where $\varphi_{\hfp}$ denotes the canonical map.

Since $X$ is a $\z$-extension $E_\infty$ DGA, there is a cofibrant commutative $\Sp$-algebra $Y$ such that $H\Z \wedge Y$ is weakly equivalent to $HX$ in commutative $H\Z$-algebras.

Note that $H\Z \wedge Y$ is cofibrant as a commutative $H\Z$-algebra; this is the case  because $H\Z \wedge - $ is a left Quillen functor from commutative $\Sp$-algebras to commutative $H\Z$-algebras and therefore it preserves cofibrant objects. 

Since $HX$ is fibrant and $H\Z \wedge  Y$ is cofibrant, there is a weak equivalence of commutative $H\Z$-algebras $  \psi  \co H\Z \wedge Y \we HX$. Because $\psi$ is a map of commutative $\hz$-algebras, we obtain a commutative diagram

\begin{equation*} \label{diag ex 3}
 \begin{tikzcd}
& H\Z \arrow[rd,"\varphi_{\hfp}"] \arrow[ldd,swap,"\varphi_{H\Z \wedge Y}"]& 
\\
&&\hfp \ar[rd]&\\
 H\Z \wedge Y \arrow[rrr,"\simeq",swap] \arrow[rrr,"\psi"]
 & & & HX
 \end{tikzcd}
 \end{equation*}
where the composite on the right from $\hz$ to $HX$ is the composite given above. The map   $\varphi_{H\Z \wedge Y}$ is the $H\Z$-structure map of $H\Z \wedge Y$ which is given by $H\Z \cong H\Z \wedge \Sp \to H\Z \wedge Y$. 

Applying the homology functor ${H\F_p}_*$ to this diagram and inverting $\hfps \psi$, we obtain the following.
\begin{equation*} 
 \begin{tikzcd}
& {H\F_p}_* H\Z \arrow[rd,"{H\F_p}_* \varphi_{\hfp}"] \arrow[ldd,swap,"{H\F_p}_*\varphi_{H\Z \wedge Y}"]& & 
\\
& & \hfps \hfp \ar[dr]&\\
 {H\F_p}_* H\Z \otimes {H\F_p}_* Y 
 & & & \hfps HX \arrow[lll,"\cong",swap] \arrow[lll,"({H\F_p}_* \psi)^{-1}"]
 \end{tikzcd}
 \end{equation*}
 
 By the Künneth spectral sequence in \cite[\rom{4}.4.1]{ekmm}, ${H\F_p}_* ( H\Z \wedge Y ) \cong {H\F_p}_* H\Z \otimes {H\F_p}_* Y$ and the morphism on the left is given by
 \begin{equation}\label{eq varphihzy a is a ot 1}
 {H\F_p}_*\varphi_{H\Z \wedge Y} (a) = a \otimes 1. 
  \end{equation}
 
 Since the diagram above commutes, we obtain that ${H\F_p}_*\varphi_{H\Z \wedge Y} $ factors as 
 \begin{equation}\label{eq composite for varphi hz wdg y}
     {H\F_p}_*\varphi_{H\Z \wedge Y} \co \hfps \hz \xrightarrow{\hfps \varphi_{\hfp}}\hfps \hfp \xrightarrow{f} {H\F_p}_* H\Z \otimes {H\F_p}_* Y
 \end{equation}
 where the second map $f$ is the composite in the triangle above starting from $\hfps \hfp$ and ending in the  bottom left corner. Both maps in the composite above are ring maps that preserve the Dyer-Lashof operations.

Let $p$ denote an odd prime, we discuss the case $p=2$ at the end of this proof. We have  $\beta \mathrm{Q^1}\tau_0 = \zeta_1$ (up to a unit we are going to omit) in $H{\F_p}_* H\F_p$. 
Note that $f(\zeta_1) = \zeta_1 \otimes 1$. This follows by considering the composite in \eqref{eq composite for varphi hz wdg y}, Equation \eqref{eq varphihzy a is a ot 1}  and by noting that $\hfps \varphi_{\hfp}$ is the canonical inclusion. Since $f$ preserves Dyer-Lashof operations, we obtain the following. 
\[ \beta \mathrm{Q^1}f(\tau_0) = f(\beta \mathrm{Q^1}\tau_0) = f(\zeta_1) = \zeta_1\otimes 1\]
We conclude that $\beta \mathrm{Q^1} f(\tau_0) = \zeta_1\otimes 1$ in $H{\F_p}_*H\Z \otimes H{\F_p}_* Y$.

 We obtain a contradiction by showing that there is no $z$ in $H{\F_p}_*H\Z \otimes H{\F_p}_* Y$ that satisfies $\beta \mathrm{Q^1} z = \zeta_1\otimes 1$, i.e.\ there is no candidate for $f(\tau_0)$. For an element of the form $1 \otimes y \in H{\F_p}_*H\Z \otimes H{\F_p}_* Y$, $\beta \mathrm{Q^1} (1\otimes y) = 1 \otimes \beta \mathrm{Q^1} y$ does not contain $\zeta_1 \otimes 1$ as a summand. Now consider an element of the form $a \otimes y \in H{\F_p}_*H\Z \otimes H{\F_p}_* Y$ with $\lvert a \rvert >0$. By the Cartan formula and the fact that the Bockstein operation  is a derivation, 
   $\beta \mathrm{Q^1} (a \otimes y)$ is a sum of elements of the form $a^{\prime} \otimes y^{ \prime}$ where $a^{'}$ is obtained by applying a Dyer-Lashof operation to $a$. In particular, $\lvert a^{\prime} \rvert > \lvert a \rvert \geq \lvert \zeta_1 \rvert$, therefore $\beta \mathrm{Q^1} (a \otimes y)$ does not contain $\zeta_1 \otimes 1$ as a summand neither. We deduce that $\beta \mathrm{Q^1}z$ does not contain $\zeta_1 \otimes 1$ as a summand for all $z \in H{\F_p}_*H\Z \otimes H{\F_p}_* Y$.
 
 For $p=2$, we do not need to use the Dyer-Lashof operations. In this case, we have $f(\zeta_1^2) = \zeta_1^2 \otimes 1$ due to the composite in \eqref{eq composite for varphi hz wdg y}. We obtain that $f(\zeta_1)^2 = \zeta_1^2 \otimes 1.$ However, there is no element in $H{\F_2}_* H\Z \otimes H{\F_2}_* Y$ that squares to $\zeta_1^2 \otimes 1$. Since this does not use Dyer-Lashof operations, these arguments also work for DGAs and $H\Z$-algebras and provide a proof of Theorem \ref{thm nonext for p is 2}.
\end{proof}

\section{Formal DGAs to $H\Z$-algebras} \label{sec formaldgas to hzalgebras}
This section is devoted to the proof of Proposition \ref{prop formal hzalg} which provides an explicit  description of the $HR$-algebra corresponding to a formal $R$-DGA whose homology satisfies the hypothesis of Theorem \ref{thm formal are extension}. This  description  provides Theorem \ref{thm formal are extension}. Recall that we also use Proposition \ref{prop formal hzalg} to obtain Corollary \ref{prop formal topological equivalences}. 

We work in several different monoidal categories in this section. When we work in the category of chain complexes or in the category of differential graded algebras, we denote the monoidal product by $\otimes$. In all the other cases, we let $\wedge$ denote the monoidal product. In particular, we contradict the notation we used in the previous sections and denote the monoidal product of $HR$-modules by $\wedge$ instead of $\wedge_{HR}$. In this section, $HR$ denotes the Eilenberg Mac Lane spectrum of a general discrete commutative ring as in \cite[1.2.5]{hovey2000symmetricspectra}. 

Let $X$ be an $R$-DGA satisfying the hypothesis of Theorem \ref{thm formal are extension}. Recall from Remark \ref{rem on graded monoid rings} that  there is a monoid $M$ in graded pointed sets for which  $H_*(X) \cong R \langle M \rangle$ as  $R$-algebras where the underlying $R$-module of $ R \langle M \rangle$ is the free graded $R$-module over  the graded set  $M_{-}$ obtained by removing the base point of $M$. Furthermore, the multiplication on $R\langle M \rangle$ is the canonical one induced by that of $M$. For the rest of this section, let $M$ denote a monoid in non-negatively graded pointed sets. 

\subsection{A monoid object corresponding to $M$}

Here, we construct a monoid in a general monoidal category by using $M$. Furthermore, we show that this construction is preserved by  strong monoidal Quillen Pairs.

We start by explaining a notation we use for the symmetric monoidal pointed model categories we consider in this section. For a cofibrant $C$, $\Sigma C$ denotes the pushout of the diagram $\Bar{*} \lcof C \cof \Bar{*}$ where $\Bar{*}$ is obtained by a factorization $C \cof \Bar{*} \trfib *$ of the map $C \to *$ by a cofibration followed by a trivial fibration and $*$ denotes the final object. For the unit $\mathbb{I}$ of the monoidal structure, $\Sigma^{k} \mathbb{I}$ denotes $(\Sigma{\mathbb{I}})^{\wedge k}$ for $k>0$ and denotes $\mathbb{I}$ for $k=0$. 
\begin{construction}\label{constr for graded monoids}
Let ($\mathcal{C}$, $\wedge,$ $\mathbb{I}$) denote a pointed cofibrantly generated closed symmetric monoidal model category whose unit $\mathbb{I}$ is cofibrant. Furthermore, assume that $\mathcal{C}$ satisfies the monoid axiom and the smallness axioms of \cite{schwede2000algebrasandmodules}. This implies that the category of modules over a monoid in $\mathcal{C}$ carries an induced model structure where the fibrations and the weak equivalences are precisely  those of $\mathcal{C}$ \cite[4.1]{schwede2000algebrasandmodules}. For a given $M$ as above, we construct a monoid structure on 
$$\vee_{m \in M_{-}} \Sigma^{\lvert m \rvert} \mathbb{I}$$
where $\vee$ denotes the coproduct in $\mathcal{C}$.
The multiplication map 
\begin{equation} \label{eq monoidforformaldga}
 \left(\vee_{m \in M_{-}} \Sigma^{\lvert m \rvert} \mathbb{I} \right) \wedge \left(\vee_{n \in M}\Sigma^{\lvert n \rvert} \mathbb{I}\right)   \cong \vee_{(m,n) \in M \times M}  \Sigma^{\lvert m \rvert + \lvert n \rvert}\mathbb{I} \to \vee_{m \in M_{-}} \Sigma^{\lvert m \rvert} \mathbb{I}
\end{equation}
is given (on the cofactor corresponding to $(m,n) \in M \times M$) by the inclusion of the cofactor corresponding to $mn \in M$ if $mn \neq 0$ and given by the zero map if $mn = 0$. Note that in a pointed model category, there is a unique zero map between every pair of objects which is defined to be the map that factors through the point object.  One easily checks that the multiplication above is associative and unital. 


If $E$ is a commutative monoid in $\mathcal{C}$, then the category of $E$-modules is also a symmetric monoidal model  category \cite[4.1]{schwede2000algebrasandmodules}. We let $\vee_{m \in M_{-}} \Sigma^{\lvert m \rvert} E$ denote the monoid we obtain by applying the construction above in the category of $E$-modules. In particular, $\vee_{m \in M_{-}} \Sigma^{\lvert m \rvert} E$ is an $E$-algebra.
\end{construction}
Using the construction above, we obtain an $HR$-algebra  
$\vee_{m \in M_{-}} \Sigma^{\lvert m \rvert} HR$. In order to prove Theorem \ref{thm formal are extension}, we go through the zig-zag of Quillen equivalences between the model categories of $R$-DGAs and $HR$-algebras to show that  the $HR$-algebra corresponding to the formal $R$-DGA with homology $R\langle M \rangle$ is given by $\vee_{m \in M_{-}} \Sigma^{\lvert m \rvert} HR$. We deduce that the formal $R$-DGA with homology $R \langle M \rangle$ is $R$-extension by showing that $\vee_{m \in M_{-}} \Sigma^{\lvert m \rvert} HR$ is weakly equivalent to $HR \wdg c(\vee_{m \in M_{-}} \Sigma^{\lvert m \rvert} \sph) $ in $HR$-algebras where $c$ denotes the cofibrant replacement functor in $\sph$-algebras. For this, we start with the following lemmas.

\begin{lemma}\label{lem weak equivalence between suspensions}
Assume that  $(\mathcal{C}, \wedge, \ic )$ and $(\dd,\wdg, \id)$ are pointed and closed symmetric monoidal model categories with cofibrant units. Furthermore, let 
\begin{equation*}
    \begin{tikzcd}
    & \cc \arrow[r, "F", shift left]
 & \arrow [l,"G", shift left] \dd  
    \end{tikzcd}
\end{equation*}
be a Quillen pair where $F$ denotes the left adjoint. If there is a weak equivalence $\upsilon \co F(\ic) \we \id$, then there exists a weak equivalence \[\varphi \co F( \Sigma \ic) \we \Sigma \id.\]
\end{lemma}
\begin{proof}
By factoring  the map $\ic \to *$ by a cofibration followed by a trivial fibration, we obtain a factorization $ F(\ic)   \cof F(\Bar{*}) \we F(*) \cong *$. Note that the  isomorphism follows by the fact that $F$ is a left adjoint functor between pointed categories. To see that the second map is a weak equivalence, note that $*$ is cofibrant in the pointed model category $\cc$   and that $F$ preserves all weak equivalences between cofibrant objects. 
Similarly, we have a factorization $\id \cof \Bar{*} \trfib *$ consisting of a cofibration followed by a trivial fibration. We use the equivalence  $\upsilon \co F(\ic) \we \id$ and the lift in the following square
\begin{equation*}
    \begin{tikzcd}
     F(\ic) \arrow[r,rightarrow,"\sim",swap] \ar[r,"\upsilon"] \arrow[d,rightarrowtail] & \id \arrow[r,rightarrow]& \Bar{*} \arrow[d,twoheadrightarrow,"\sim"]
    \\
     F(\Bar{*}) \arrow[urr,rightarrow,dashed,"\sim"] \arrow[rr,,"\sim",rightarrow]& &  F(*) \cong *
    \end{tikzcd}
\end{equation*}
to obtain a weak equivalence of diagrams 
\[\begin{tikzcd}[sep=small] (F(\Bar{*})  \arrow[r,leftarrowtail]& F(\ic) \arrow[r,rightarrowtail]& F(\Bar{*})) \arrow[r,"\sim",rightarrow] &  (\Bar{*} \arrow[r,leftarrowtail]& \id \arrow[r,rightarrowtail]& \Bar{*}).\end{tikzcd}\] This in turn gives a map $\varphi$ of the corresponding pushouts of these diagrams. This is a weak equivalence because these are diagrams consisting only of cofibrations between cofibrant objects; therefore their pushout is the homotopy pushout. Since the pushout of the diagram on the left hand side  is $F(\Sigma \ic)$ and the pushout of the diagram on the right hand side is $\Sigma \id$, we obtain  the weak equivalence
\[\varphi \co F(\Sigma \ic) \we   \Sigma \id\]
we wanted to construct.

\end{proof}

\begin{lemma}\label{lem strong monoidal preserve the monoid construction}
Assume that  $(\mathcal{C}, \wedge, \ic )$ and $(\dd,\wdg, \id)$ are pointed and closed symmetric monoidal model categories with cofibrant units as in Construction \ref{constr for graded monoids}. Furthermore, let 
\begin{equation*}
    \begin{tikzcd}
    & \cc \arrow[r, "F", shift left]
 & \arrow [l,"G", shift left] \dd  
    \end{tikzcd}
\end{equation*}
be a Quillen pair where the left adjoint $F$ is a strong monoidal functor. In this situation,  $Fc(\vee_{m \in M_{-}} \Sigma^{\lvert m \rvert} \ic)$ and $\vee_{m \in M_{-}} \Sigma^{\lvert m \rvert} \id$ are weakly equivalent as monoids in $\dd$ where $c$ denotes the cofibrant replacement functor in the model category of monoids in $\cc$ \cite[4.1]{schwede2000algebrasandmodules}.
\end{lemma}
\begin{proof}
 Since $F$ is a strong monoidal functor, we have a natural isomorphism 
$F(X) \wedge F(Y) \cong F(X\wedge Y)$ and an isomorphism $F(\ic) \cong \id$. This isomorphism provides the weak equivalence $\upsilon$ in the hypothesis of Lemma \ref{lem weak equivalence between suspensions}. Therefore, there is a weak equivalence  $\varphi \co F(\Sigma \ic) \we   \Sigma \id$. 
 
Using $\varphi$, we produce a weak equivalence of monoids:
\[\begin{tikzcd}\Phi \co F(\vee_{m \in M_{-}} \Sigma^{\lvert m \rvert} \ic)\cong \vee_{m \in M_{-}} F(\Sigma^{\lvert m \rvert} \ic) \arrow[r,rightarrow,"\sim"] &\vee_{m \in M_{-}} \Sigma^{\lvert m \rvert} \id .\end{tikzcd}\] Here, $\Phi$ is the coproduct of maps given by the isomorphism $F(\ic) \cong \id$ for $\lvert m \rvert = 0$ and the map
\[\begin{tikzcd} F(\Sigma^{\lvert m \rvert} \ic) = F((\Sigma \ic)^{\wedge \lvert m \rvert}) \cong F(\Sigma \ic)^{\wedge \lvert m \rvert}  \arrow[r,swap,"\sim"] \arrow[r,"\varphi^{\wedge \lvert m \rvert}"] & (\Sigma \id)^{\wedge \lvert m \rvert} = \Sigma^{\lvert m \rvert} \id \end{tikzcd}\] for $\lvert m \rvert>0$. Here, the first and the last equalities follow by our definition of $\Sigma^k-$ for $k>0$ and the second isomorphism comes from the strong monoidal structure of $F$. Also, note that  $\varphi^{\wedge \lvert m \rvert}$ is a weak equivalence because it is a smash product of weak equivalences between cofibrant objects. Since $\Phi$ is a coproduct of weak equivalences between cofibrant objects, it is a weak equivalence by Lemma 4.7 of \cite{White2017modelstructuresoncommutative}. It is clear that $\Phi$ is a map of monoids by the definition of the monoidal structure on both sides and from the fact that left adjoint functors between pointed categories  preserve the zero maps. This shows that $\Phi$ is a weak equivalence of monoids between $F(\vee_{m \in M_{-}} \Sigma^{\lvert m \rvert} \ic)$ and $\vee_{m \in M_{-}} \Sigma^{\lvert m \rvert} \id$.

Therefore, in order to  finish the proof of the  lemma, it is sufficient to show that the monoids $Fc(\vee_{m \in M_{-}} \Sigma^{\lvert m \rvert} \ic)$ and $F(\vee_{m \in M_{-}} \Sigma^{\lvert m \rvert} \ic)$ are weakly equivalent. Since $c$ is the cofibrant replacement functor in the category of monoids, there is a weak equivalence of monoids
\[\begin{tikzcd} f \co c(\vee_{m \in M_{-}} \Sigma^{\lvert m \rvert} \ic) \arrow[r,rightarrow,"\sim"] & \vee_{m \in M_{-}} \Sigma^{\lvert m \rvert} \ic. \end{tikzcd}\] By Theorem 4.1 of \cite{schwede2000algebrasandmodules}, the source of $f$ is cofibrant in $\cc$. This means that $f$ is a weak equivalence between cofibrant objects and therefore $F(f)$ is a weak equivalence. Furthermore, $F(f)$ is a weak equivalence of monoids because a strong monoidal functor preserves maps of monoids. Therefore, the monoids $Fc(\vee_{m \in M_{-}} \Sigma^{\lvert m \rvert} \ic)$ and $F(\vee_{m \in M_{-}} \Sigma^{\lvert m \rvert} \ic)$ are weakly equivalent as desired.

\end{proof}

\subsection{From DGAs to $H\Z$-algebras}

Here, we carry out our discussion for the case $R=\Z$. The case of general discrete commutative ring $R$ follows similarly. 

The DGA corresponding to an $H\Z$-algebra is obtained using the following zig-zag of monoidal Quillen equivalences of  \cite{Shipley2007DGAs}
\begin{equation*} \label{diag zigzag of Shipley}
 \begin{tikzcd}
& H\Z \mbox{-}  \mathcal Mod \arrow[r, "Z", shift left]
 & \arrow [l,"U", shift left] Sp^{\Sigma}(s \mathcal AB) \arrow[r,"\phi^{*}N",swap, shift right] &  \arrow [l,"L",swap, shift right] Sp^{\Sigma}(\mathcal Ch^+) \arrow[r,"D", shift left]&  \arrow [l,"R", shift left]  \mathcal Ch 
  & 
 \end{tikzcd}
 \end{equation*}
where the left adjoints are the top arrows and the pairs $(Z,U)$ and $(D,R)$ are both strong monoidal Quillen equivalences. The pair $(L,\phi^*N)$ is a weak monoidal Quillen equivalence. See \cite[3.6]{Schwede2003monoidalquillen} for the definitions of strong monoidal Quillen equivalences and weak monoidal Quillen equivalences. We often use the fact that the model categories in the zig-zag above are pointed. 

Since each Quillen equivalence in the zig-zag is a monoidal Quillen equivalence, there is an induced  zig-zag of Quillen equivalences of the corresponding model categories of monoids. This gives the induced derived functors $\mathbb{H} \co \DGA \to H\Z \alg$ and $\Theta \co H\Z \alg \to \DGA$ in Theorem 1.1 of \cite{Shipley2007DGAs}. We have 
 \[\Theta = Dc \phi^* N Zc  \ \ \ \ \ \mathbb{H} = U L^{mon} cR \]
where $L^{mon}$ is the induced left adjoint at the level of monoids and $c$ denotes the cofibrant replacement functors in the corresponding model category of monoids. See Section 3.3 of \cite{Schwede2003monoidalquillen} for a definition of the induced left adjoint at the level of monoids.

In the lemmas below, $\mathbb{I}_1$ and $\I_2$ denote the monoidal units of $Sp^{\Sigma}(s \mathcal AB)$ and $Sp^{\Sigma}(\mathcal Ch^+)$ respectively. 
Note that the units of the monoidal model categories in the zig-zag above are all cofibrant. By Construction \ref{constr for graded monoids}, we have the monoids  $\vee_{m \in M_{-}} \Sigma^{\lvert m \rvert} \I_1$ and $\vee_{m \in M_{-}} \Sigma^{\lvert m \rvert} \I_2$ in $Sp^{\Sigma}(s \mathcal AB)$ and $Sp^{\Sigma}(\mathcal Ch^+)$ respectively.  
\begin{lemma} \label{lemma first and last quillen equivalences}
 In $Sp^{\Sigma}(s \mathcal AB)$, $Zc(\vee_{m \in M_{-}} \Sigma^{\lvert m \rvert} H\Z)$ and $\vee_{m \in M_{-}} \Sigma^{\lvert m \rvert} \I_1$ are weakly equivalent as monoids. In $\mathcal Ch$, $Dc(\vee_{m \in M_{-}} \Sigma^{\lvert m \rvert} \I_2)$ and the formal DGA with homology $\Z \langle M \rangle$ are quasi-isomorphic as DGAs.
\end{lemma}

\begin{proof}
The first statement is a direct consequence of Lemma \ref{lem strong monoidal preserve the monoid construction}. We prove the second statement of the lemma. It again follows by Lemma \ref{lem strong monoidal preserve the monoid construction}    that $Dc(\vee_{m \in M_{-}} \Sigma^{\lvert m \rvert} \I_2)$ and $\oplus_{m \in M_{-}} \Sigma^{\lvert m \rvert} \Z$ are quasi-isomorphic as DGAs (i.e.\ weakly equivalent as monoids in $\mathcal Ch$). 

Therefore, it is sufficient to show that  $\oplus_{m \in M_{-}} \Sigma^{\lvert m \rvert} \Z$ is quasi-isomorphic to the  formal DGA with homology $\Z \langle M \rangle$. Let $\Bar{0}$ denote the chain complex consisting of $\Z$ in degrees $0$ and $1$ and the trivial module in the rest of the degrees; its differentials  are trivial except degree 1 where the differential is the identity. There is a factorization $\Z \cof \Bar{0} \trfib 0$  of the trivial map $\Z \to 0$ as a cofibration followed by a trivial fibration.

Let $\sigma \Z$ denote the chain complex consisting of $\Z$ in degree $1$ and the trivial module in rest of the degrees. This is the pushout of the diagram $ \Bar{0} \lcof \Z \to 0$.  

Note that due to our conventions, $\Sigma \Z$ is the pushout of the diagram $ \Bar{0} \lcof \Z \to \Bar{0}$. Since the category of chain complexes is left proper, there is a weak equivalence $ \varphi \co \Sigma \Z   \we \sigma \Z$. Let $\sigma^n \Z$  denote  $(\sigma^1 \Z)^{\otimes n}$. Following Construction \ref{constr for graded monoids}, we obtain a formal DGA $\oplus_{m \in M_{-}} \sigma^{\lvert m \rvert} \Z$.   Similar to the map $\Phi$ in the proof of Lemma \ref{lem strong monoidal preserve the monoid construction}, we  obtain a quasi-isomorphism of DGAs
\[\Phi \co \oplus_{m \in M_{-}} \Sigma^{\lvert m \rvert} \Z  \we \oplus_{m \in M_{-}} \sigma^{\lvert m \rvert} \Z\]
given by the identity map for $\lv m \rv = 0$ and given by $\varphi^{\lv m \rv}$ for $\lv m \rv >0$. This shows that 
 $\oplus_{m \in M_{-}} \Sigma^{\lvert m \rvert} \Z$ and $\oplus_{m \in M_{-}} \sigma^{\lvert m \rvert} \Z$ are quasi-isomorphic as DGAs where the latter is the formal DGA with homology $\Z \langle M \rangle$.
\end{proof}

We state and prove the following two lemmas that we use in the proof of Lemma \ref{leamma middle quillen equivalence}.

\begin{lemma} \label{lemma preserve coproducts}
The functor $\phi^*N$ preserves colimits. 
\end{lemma}

\begin{proof}
The category of symmetric spectra in a closed symmetric monoidal model category $\mathcal C$ is the category of  modules over a  monoid in symmetric sequences in $\mathcal C$, see Definition 2.7 in \cite{Shipley2007DGAs}. Since symmetric sequences in $\mathcal C$ is a diagram category in $\mathcal C$, the colimits in symmetric sequences  are levelwise. Furthermore, the forgetful functor from modules over a monoid to the underlying closed monoidal category preserves colimits. Therefore colimits of  symmetric spectra in $\mathcal C$ are also levelwise. 

Here, $N$ is the normalization functor $s\mathcal AB \to \mathcal Ch^+$ of the Dold-Kan correspondence, an equivalence of categories, applied levelwise. Therefore it preserves colimits. Furthermore, $\phi^*$ is the restriction of scalars functor between the categories of modules over two monoids induced by a map of these monoids in symmetric sequences in $\mathcal Ch^+$, see Page 358 of \cite{Shipley2007DGAs}. Therefore $\phi^*$ is the identity functor identity on the underlying symmetric sequences and therefore it also preserves colimits.
\end{proof}
\begin{lemma}\label{lemma adjointweakequivalences}
For every cofibrant $A$ in $Sp^{\Sigma}(\mathcal Ch^+)$ and every $B$ in $Sp^{\Sigma}(s \mathcal AB)$, a map $L(A) \to B$ is a weak equivalence if and only if its adjoint $A \to \phi^*N(B)$ is a weak equivalence. 

\end{lemma}
\begin{proof}
This follows from the fact that $\phi^*N$ preserves all weak equivalences. Let $B \we fB$ be a fibrant replacement of $B$. Adjoint of the composite $L(A) \to B \we fB$ is given by the composite $A \to \phi^*N(B) \we \phi^*N(fB)$ whose first map is the adjoint of the map $L(A)\to B$. Because $(L,\phi^*N)$ is a Quillen equivalence, the first composite is a weak equivalence if and only if the second composite is a weak equivalence. The result follows by the 2-out-of-3 property of weak equivalences.
\end{proof}

The following lemma takes care of the middle step in the zig-zag of Quillen equivalences between the model categories of  $\hz$-algebras and DGAs. Note that since $(L,\phi^*N)$ is a weak monoidal Quillen pair, $\phi^*N$ is a lax monoidal functor, see Definition 3.3 of \cite{Schwede2003monoidalquillen}.  Therefore, $\phi^*N$ carries monoids to monoids. In particular, $\phi^* N(\vee_{m \in M_{-}} \Sigma^{\lvert m \rvert} \I_1)$ is a monoid.

\begin{lemma}\label{leamma middle quillen equivalence}
In $Sp^{\Sigma}(\mathcal Ch^+)$, $\phi^* N(\vee_{m \in M_{-}} \Sigma^{\lvert m \rvert} \I_1)$ and $\vee_{m \in M_{-}} \Sigma^{\lvert m \rvert} \I_2$ are weakly equivalent as monoids. 
\end{lemma}

\begin{proof}
 By Lemma \ref{lemma preserve coproducts}, $\phi^*N$ preserves coproducts. Therefore, there is an isomorphism
\begin{equation}\label{eq iso of monoids phi star n}
\phi^* N(\vee_{m \in M_{-}} \Sigma^{\lvert m \rvert} \I_1) \cong \vee_{m \in M_{-}}\phi^* N( \Sigma^{\lvert m \rvert} \I_1).
\end{equation}
Similar to Construction \ref{constr for graded monoids}, the object on the right hand side above carries a canonical monoid structure given by the multiplication on $M$ and the lax monoidal structure of $\phi^*N$. Namely, the multiplication map 
\begin{equation*} 
 \vee_{m \in M_{-}}\phi^* N( \Sigma^{\lvert m \rvert} \I_1) \wedge \vee_{n \in M_{-}}\phi^* N( \Sigma^{\lvert n \rvert} \I_1) \to \vee_{m \in M_{-}}\phi^* N( \Sigma^{\lvert m \rvert} \I_1)
\end{equation*}
is given (on the cofactor corresponding to $(m,n) \in M \times M$) by the composite
\[\phi^* N( \Sigma^{\lvert m \rvert} \I_1) \wdg \phi^* N( \Sigma^{\lvert n \rvert} \I_1) \to \phi^* N( \Sigma^{\lvert m \rvert} \I_1 \wdg  \Sigma^{\lvert n \rvert} \I_1) = \phi^* N( \Sigma^{\lvert mn \rvert} \I_1) \]
followed by the inclusion of the cofactor corresponding to $mn \in M$ if $mn \neq 0$ and given by the zero map if $mn = 0$. Note that the  map above is the lax monoidal structure map of $\phi^*N$ and the equality above follows by our definition of  $\Sigma^k -$. Furthermore, one checks using this definition that the isomorphism in \eqref{eq iso of monoids phi star n}  is an isomorphism of monoids. Therefore, in order to prove the lemma, it is sufficient to show that there is an isomorphism of monoids between $\vee_{m \in M_{-}}\phi^* N( \Sigma^{\lvert m \rvert} \I_1)$ and $\vee_{m \in M_{-}} \Sigma^{\lvert m \rvert} \I_2 $.

There is a weak equivalence  $L(\I_2) \we  \I_1$, see  \cite[3.6]{Schwede2003monoidalquillen}. Therefore, there is also a weak equivalence $ \varphi \co L(\Sigma \I_2) \we  \Sigma \I_1$ by Lemma \ref{lem weak equivalence between suspensions}. Let
\[\psi \co \Sigma \I_2 \to \phi^*N(\Sigma \I_1)\]
be the adjoint of $\varphi$. 

Let $\psi^0$ denote the unit $\I_2 \to \phi^*N(\I_1)$ of the lax monoidal structure  of $\phi^*N$ and let $\psi^1$ denote $\psi$. For $\ell>1$, let $\psi^\ell$ denote the composite:
\[\psi^{ \ell} \co \Sigma^\ell \I_2 = (\Sigma \I_2)^{\wedge \ell} \xrightarrow{\psi^{\wdg \ell}} (\phi^*N (\Sigma \I_1))^{\wedge \ell} \to \phi^*N((\Sigma \I_1)^{\wedge \ell}) = \phi^*N(\Sigma^\ell \I_1), \]
where the equalities follow  by our definition of $\Sigma^\ell$ and the second map is obtained by successive applications of the  transformation $\phi^*N(-) \wedge \phi^*N(-) \to \phi^*N(- \wedge -)$ that is a part of the lax monoidal structure of $\phi^*N$, see  \cite[3.3]{Schwede2003monoidalquillen}.

Now we define a map of monoids
\[\Psi: \vee_{m \in M_{-}} \Sigma^{\lvert m \rvert} \I_2 \to  \vee_{m \in M_{-}}\phi^* N( \Sigma^{\lvert m \rvert} \I_1)\]
 as the coproduct of $\psi^{\lvert m \rvert}$ over $m \in M_{-}$.  By the associativity and the unitality of the lax monoidal structure on $\phi^*N$ and by the fact that right adjoint functors preserve the zero maps between pointed categories, $\Psi$ is a map of monoids, see  \cite[6.4.1]{borceux1994handbook2}.

Finally, we need to show that $\Psi$ is a weak equivalence. By Lemmas \ref{lemma preserve coproducts} and \ref{lemma adjointweakequivalences}, it is sufficient to show that the adjoint of $\Psi$ is a weak equivalence. Since both $\phi^*N$ and $L$ preserve coproducts and since $\Psi$ is a coproduct of maps $\psi^{\lvert m \rvert}$, the adjoint of $\Psi$ is a coproduct of the adjoints of the maps $\psi^{\lvert m \rvert}$. Note that a coproduct of weak equivalences of cofibrant objects is again a weak equivalence  by  \cite[4.7]{White2017modelstructuresoncommutative}. Since the adjoint of $\psi^{\ell}$ is a map between cofibrant objects, it is sufficient to show that the adjoint of $\psi^{\ell}$ is a weak equivalence for each $\ell \geq 0$.

For the case $\ell=0$, we have that the adjoint of $\psi^0$ 
is the weak equivalence $L(\I_2) \we \I_1$ mentioned above. For $\ell=1$, the adjoint of  $\psi^1$ is  the map $\varphi$  above which is also a weak equivalence. 

We show the $\ell=2$ case and the rest follows similarly. In particular, we show that the adjoint to the composite defining $\psi^2$ 
\begin{equation*}
\begin{tikzcd}
\psi^2 \co \Sigma \I_2 \wedge \Sigma \I_2 \ar[r,rightarrow,"\psi \wedge \psi"] & \phi^*N(\Sigma \I_1) \wedge \phi^*N(\Sigma \I_1) \to \phi^*N(\Sigma \I_1 \wedge \Sigma \I_1)
\end{tikzcd}
\end{equation*}
is the composite map 
\begin{equation} \label{compositemap}
\begin{tikzcd}
   L(\Sigma \I_2 \wedge \Sigma \I_2) \ar[r,rightarrow,"c_L"]\ar[r,swap,"\simeq"] &\ar[r,swap,"\simeq"] L(\Sigma \I_2) \wedge L(\Sigma \I_2) \ar[r,rightarrow,"\varphi \wedge \varphi"] & \Sigma \I_1 \wedge \Sigma \I_1.
   \end{tikzcd}
\end{equation}
The first map in this composite is the comonoidal map induced by the lax monoidal structure of $\phi^*N$ and by Definition 3.6 of \cite{Schwede2003monoidalquillen}, this is a weak equivalence. Furthermore, the second map in the composite is a smash product of weak equivalences between cofibrant objects; and therefore, it is also a weak equivalence. This shows that the composite is a weak equivalence. 

To show that $\psi^2$ is the adjoint to this composite, first note that by the discussion on Equation (3.4) in \cite{Schwede2003monoidalquillen}, the comonoidal  map $c_L$ is the adjoint of the composite map 
\[ \Sigma \I_2 \wedge \Sigma \I_2 \to \phi^*N L(\Sigma \I_2) \wedge \phi^* N L (\Sigma \I_2) \to \phi^*N(L(\Sigma \I_2) \wedge L(\Sigma \I_2))\]
where the first map is induced by the unit of the adjunction and the second map comes from the lax monoidal structure on $\phi^*N$. Considering the adjoint of the composite \eqref{compositemap} as the adjoint of the first map $c_L$ in the composite followed by $\phi^*N(\varphi \wedge \varphi)$, we obtain that the adjoint of \eqref{compositemap} is given by the composite
\[ \Sigma \I_2 \wedge \Sigma \I_2 \to \phi^*N L(\Sigma \I_2) \wedge \phi^* N L (\Sigma \I_2) \to \phi^*N(L(\Sigma \I_2) \wedge L(\Sigma \I_2)) \to \phi^*N(\Sigma \I_1 \wedge \Sigma \I_1).\]
By the naturality of the monoidal transformation  of $\phi^*N$, this composite is equal to the canonical composite
\[ \Sigma \I_2 \wedge \Sigma \I_2 \to \phi^*N L(\Sigma \I_2) \wedge \phi^* N L (\Sigma \I_2) \to \phi^*N(\Sigma \I_1) \wedge \phi^*N( \Sigma \I_1) \to \phi^*N(\Sigma \I_1 \wedge \Sigma \I_1).\]
Note that the composition of the first two maps is the smash product of adjoints of $\varphi$ which is $\psi \wedge \psi$. Therefore, this composite is precisely the composite that defines $\psi^2$ above. This shows that the adjoint of $\psi^2$ is  the composite weak equivalence in \eqref{compositemap}.
\end{proof}

\subsection{Proof of Theorem \ref{thm formal are extension}.}

We prove the following proposition which provides an explicit description of the $HR$-algebra corresponding to the formal $R$-DGA with homology $R \langle M \rangle$. 

Recall from Remark \ref{rem on graded monoid rings} that an $R$-DGA satisfying the hypothesis of Theorem \ref{thm formal are extension} is a formal $R$-DGA with homology $R\langle M \rangle$ for some monoid $M$ in non-negatively graded pointed sets. The following proposition states that such an $R$-DGA is $R$-extension. In other words,  this  proposition implies Theorem \ref{thm formal are extension}.
\begin{proposition} \label{prop formal hzalg}
The $R$-DGA corresponding to the $HR$-algebra $\vee_{m \in M_{-}} \Sigma^{\lvert m \rvert} HR$ is the formal $R$-DGA with homology $R \langle M \rangle$. Furthermore, there is an equivalence of $HR$-algebras
\[\vee_{m \in M_{-}} \Sigma^{\lvert m \rvert} HR \simeq HR \wedge c( \vee_{m \in M_{-}} \Sigma^{\lvert m \rvert} \Sp)\]
where $c$ denotes the cofibrant replacement functor in $\sph$-algebras.
\end{proposition}

\begin{proof}
For the first statement, we discuss the case $R=\Z$, the proof for general $R$ follows similarly. The first statement is a consequence of Lemmas \ref{lemma first and last quillen equivalences} and  \ref{leamma middle quillen equivalence}.

Now we prove the second statement. Recall that $HR \wdg - $ is a symmetric monoidal functor between $\sph$-modules and $HR$-modules. Therefore, the second statement is consequence of Lemma \ref{lem strong monoidal preserve the monoid construction}.
\end{proof}

\appendix
\section{}

Here, we provide a short discussion on the compatibility of Definitions \ref{def extension} and \ref{def einfty extension}. 

If we choose our $E_\infty$ operad to be the Barrat-Eccles operad, then every $E_\infty$ $R$-DGA is at the same time an $R$-DGA, see \cite[Section 1.1.1]{Berger2004combinatorialoperadactions}. Let $X$ be an $R$-extension $E_\infty$ $R$-DGA and let $U(X)$ denote its underlying $R$-DGA. The canonical compatibility question  asks if $U(X)$ is $R$-extension as an $R$-DGA. In other words, we want to know if every $R$-extension $E_\infty$ $R$-DGA forgets to an $R$-extension $R$-DGA. 
     
     Let $HX$ denote the commutative $HR$-algebra corresponding to $X$ and let $HU(X)$ denote the $HR$-algebra corresponding to $U(X)$. For the moment, assume that $HX$ is weakly equivalent to $HU(X)$ as an $HR$-algebra. Under this assumption, we conclude that $U(X)$ is $R$-extension. To see this, let $HX\simeq HR \wdg E$ for some cofibrant commutative $\sph$-algebra $E$ and let $c$ denote the cofibrant replacement functor in $\sph$-algebras. Since cofibrant (commutative) $\sph$-algebras forget to cofibrant $\sph$-modules \cite{Shipley2004convenient,schwede2000algebrasandmodules} and since the left Quillen functor $HR \wdg - $ preserves weak equivalences between cofibrant objects, we deduce that $HR \wdg E$ is equivalent to $HR \wdg cE$ in $HR$-algebras. Hence, $HU(X)$ is weakly equivalent to $HR \wdg cE$ and therefore $U(X)$ is $R$-extension as desired.
     
     However, it is not known whether $HX$ and $HU(X)$ are weakly equivalent in $HR$-algebras. In other words, it is not known if the zig-zag of Quillen equivalences between $HR$-algebras  and $R$-DGAs in \cite{Shipley2007DGAs} is compatible with the zig-zag of Quillen equivalences between commutative $HR$-algebras and $E_\infty$ $R$-DGAs in \cite{richter2017algebraicmodel}. In conclusion, if we assume that these Quillen equivalences are compatible, then the Definitions \ref{def extension} and \ref{def einfty extension} are also compatible in the sense described above.

\end{document}